\documentclass[preprint,12pt,numbers]{elsarticle}

\usepackage{amsmath,amsfonts,amssymb,amsthm}
\usepackage{mathtools}
\usepackage{graphicx}
\usepackage{subcaption}
\usepackage{booktabs}
\usepackage{multirow}
\usepackage{algorithm}
\usepackage{algorithmic}
\usepackage{wrapfig}
\usepackage[protrusion=true,expansion=false]{microtype}
\usepackage{enumitem}
\usepackage{comment}
\usepackage{hyperref}
\usepackage{xcolor}
\usepackage{placeins}

\graphicspath{{./figs/}{./figures/}{../Figure/}{../figures/}}

\allowdisplaybreaks
\theoremstyle{plain}
\theoremstyle{definition}

\newtheorem{lemma}{Lemma}[section]
\newtheorem{theorem}{Theorem}[section]
\newtheorem{proposition}{Proposition}[section]

\theoremstyle{remark}
\newtheorem{remark}{Remark}[section]

\newcommand{\YY}[1]{\textcolor{red!60!black}{#1}}

\definecolor{AIviolet}{RGB}{140,0,160}

\journal{Computers and Mathematics with Applications}

\begin{document}

\begin{frontmatter}


\title{Deterministic Adam-Inspired Methods with Accelerated Convergence Rate}

\author[aff1]{Yaxin Yu}
\ead{yaxin_yu@stu.scu.edu.cn}

\author[aff2]{Long Chen\corref{cor1}}
\ead{lchen7@uci.edu}

\author[aff2]{Zeyi Xu}
\ead{zeyix1@uci.edu}

\affiliation[aff1]{organization={School of Mathematics, Sichuan University},
	city={Chengdu},
	state={Sichuan},
	postcode={610065},
	country={CHINA}}

\affiliation[aff2]{organization={Department of Mathematics, University of California, Irvine},
	city={Irvine},
	state={CA},
	postcode={92697},
	country={USA}}

%\cortext[cor1]{Corresponding author.}

%% Abstract
\begin{abstract}
Adam is widely used, but its convergence theory remains incomplete even in the deterministic full-batch setting because momentum and adaptive preconditioning are tightly coupled. For smooth convex objectives, we split the momentum variable through variable-and-operator splitting, which reveals the acceleration mechanism. We then combine a Hessian-driven correction with Adam-style feedback based on the gradient magnitude. The resulting Adam-HNAG (Hessian-driven Nesterov accelerated gradient with Adam-style adaptive preconditioning) flow admits a nonnegative energy that decays exponentially. Its discretization yields two methods, \emph{Adam-HNAG} and the synchronous variant \emph{Adam-HNAG-s}. Under the stated trajectory-bound and consistency conditions, both methods satisfy a discrete Lyapunov contraction. If the exact adaptive steps are accepted, this contraction gives an $O(k^{-2})$ objective-value bound. Numerical experiments illustrate their behavior. These results apply to the proposed methods, not to the original Adam recursion.
\end{abstract}


%%\item Research highlight 1
%%\item Research highlight 2

%% Keywords

\end{frontmatter}



%%%%%%%%%%%%%%%%%%%%%%%%%%%%%%%%
\section{Introduction} 
%%%%%%%%%%%%%%%%%%%%%%%%%%%%%%%%

Adam~\cite{kingma2017adammethodstochasticoptimization}, short for {\it Adaptive Moment Estimation}, is one of the most widely used optimization methods in deep learning. We study its deterministic full-batch version for the unconstrained convex optimization problem
\begin{equation}\label{eq:problem}
	\min_{x\in\mathbb{R}^d} f(x),
\end{equation}
where $f:\mathbb{R}^d\to\mathbb{R}$ is continuously differentiable and convex. Let $m_k\in\mathbb{R}^d$ be the first-moment estimate and let $V_k\in\mathbb{R}^{d\times d}$ be the diagonal matrix of second-moment estimates. The method starts from prescribed $x_0,m_0\in\mathbb{R}^d$ and a diagonal matrix $V_0\succ0$. For $u\in\mathbb{R}^d$, $\operatorname{diag}(u.^2)$ denotes the diagonal matrix with entries $u_i^2$, and $\sqrt{V_k}$ is taken entrywise. Without bias correction or numerical stabilization, the deterministic full-batch Adam iteration is
\begin{equation}\label{eq:adam}
	\left\{
	\begin{aligned}
		x_{k+1} &= x_k - \eta (\sqrt{V_k})^{-1} m_k,\\
		m_{k+1} &= \beta_1 m_k + (1-\beta_1) \nabla f(x_{k+1}), \\
		V_{k+1} &= \beta_2 V_k + (1-\beta_2)\operatorname{diag}(\nabla f(x_{k+1}).^2),
	\end{aligned}
	\right.
\end{equation}
where $\eta>0$ is the learning rate and $\beta_1,\beta_2\in(0,1)$ are momentum parameters. Since $V_0\succ0$ and $\beta_2>0$, every $V_k$ remains positive definite.

%With the usual zero-initialized Adam indexing, bias correction can be implemented without changing the structural form of the preconditioned step by using the iteration-dependent effective learning rate
% $$\tilde{\eta}_k=\frac{\sqrt{1-\beta_2^k}}{1-\beta_1^k}\,\eta,\qquad k\ge 1.$$
% This scalar rescaling is omitted in \eqref{eq:adam} because it does not change the structural coupling between momentum and diagonal adaptive preconditioning studied below.
%\AI{With \(\varepsilon>0\), exact equivalence also requires the stabilization term to be rescaled. This detail does not affect the new schemes.}

Despite its empirical success, the convergence theory of Adam remains incomplete. Existing guarantees mainly concern online settings. In online convex optimization, one studies the regret
\begin{equation}\label{eq:online-regret}
R_T:=\sum_{t=1}^T f_t(x_t)-\min_x \sum_{t=1}^T f_t(x),
\end{equation}
and obtains sublinear bounds such as $R_T=O(\sqrt{T})$ under suitable assumptions~\cite{kingma2017adammethodstochasticoptimization, reddi2019convergenceadam,Huang_2019}. However, such regret guarantees do not imply convergence of the iterates for problem  \eqref{eq:problem} minimizing a single convex objective. Indeed, Reddi et al.~\cite{reddi2019convergenceadam} gave a synthetic convex counterexample showing that Adam can fail in that setting.
In nonconvex optimization, several works prove convergence of gradient norms under additional assumptions on stepsizes and momentum parameters~\cite{de2018convergenceguaranteesrmspropadam,zou2019sufficientconditionconvergencesadam, chen2019convergenceclassadamtypealgorithms, défossez2022simpleconvergenceproofadam, zhou2024convergenceadaptivegradientmethods,zhang2022adamconverge}. Online regret bounds control cumulative losses for time-varying objectives, whereas nonconvex analyses typically establish convergence to stationary points. Thus, they do not directly yield the objective-value bound considered here for a single smooth convex function. Deterministic batch Adam has also been studied through local convergence, limit cycles, and sharp local rates near strict minimizers~\cite{bock2019local,bock2022cycles,dereich2025sharp}. These results are important, but they do not establish a global $O(k^{-2})$ objective-value bound for the methods developed here. Conversely, our results do not analyze the original Adam recursion.

Continuous-time models provide a complementary perspective on Adam and related adaptive methods; see, for example, \cite{barakat2020convergencedynamicalbehavioradam, dereich2025odeapproximationadamalgorithm, Bhattacharjee_2024, ma2021qualitativestudydynamicbehavior, gould2024continuoustimeanalysisadaptiveoptimization, li2018stochasticmodifiedequationsdynamics, malladi2024sdesscalingrulesadaptive, heredia2026adamadamlikelagrangianssecondorder, heredia2025modelingadagradrmspropadam}. These works clarify the dynamics, stability, and limiting behavior of adaptive methods, but they do not yield a simple route to the global estimate proved below. Moreover, Da Silva et al.~\cite{dasilva2019generaldifferentialequationsmodel} show that, even in continuous time, Adam may converge more slowly than non-adaptive methods. 

The main obstacle is structural. Adam couples momentum and adaptive preconditioning in a highly nonlinear way. This coupling obscures the dissipative mechanism in both discrete and continuous time and makes it difficult to identify a Lyapunov functional from the standard formulation. In this paper, we start from a deterministic continuous-time Adam model and modify its dynamics so that the resulting method admits a Lyapunov analysis and provable convergence. We do not address the additional difficulties caused by stochastic gradients. Instead, we focus on the interaction between adaptivity and acceleration.

The main mathematical contributions are as follows: (i) we identify the mechanism underlying the proposed Adam-type dynamics. Splitting the momentum variable into two components reveals acceleration within the variable-and-operator-splitting (VOS) framework~\cite{chen2025acceleratedgradientmethodsvariable}. The resulting convergence rate is conditional on the trajectory bound, consistency condition, and accepted-step condition stated in the theorems and propositions. (ii) We introduce an explicit Hessian-driven correction term. Discretizing the HNAG flow yields two accelerated Adam-type methods, depending on whether the gradient term is preconditioned by $P_k^{-1}$ or $P_{k+1}^{-1}$. For both methods, we establish a product Lyapunov contraction and an $O(k^{-2})$ rate when the exact adaptive steps are accepted. (iii) We introduce feedback based on the gradient magnitude. Its square-root scaling follows from the cancellation required in the Lyapunov estimate. Although HNAG methods have been studied previously~\cite{chen2019orderoptimizationmethodsbased}, earlier Hessian-driven acceleration analyses do not include this feedback mechanism, which distinguishes Adam from HNAG.

Next, we introduce notation and preliminary results. We use the standard inner product $\langle\cdot,\cdot\rangle$ and norm $\|\cdot\|$. For a symmetric matrix $A$, we define $\langle x,y\rangle_A:=\langle Ax,y\rangle$ and $\|x\|_A^2:=\langle x,x\rangle_A$. We write $A\succ 0$ to denote a symmetric positive definite matrix. For a diagonal matrix $D=\mathrm{diag}\{d_1,\dots,d_d\}$, $D\succ 0$ means $d_i>0$ for all $i$.

For a differentiable function $f$, the Bregman divergence is
\begin{equation}\label{bregman}
	D_f(y,x):=f(y)-f(x)-\langle\nabla f(x),y-x\rangle.
\end{equation}
Let $\mathcal{S}_L$ be the set of convex, $L$-smooth differentiable functions. For any $f\in\mathcal{S}_L$, we have the bound
\begin{equation}\label{bregmanbound}
	\frac{1}{2L}\|\nabla f(x)-\nabla f(y)\|^2 \le D_f(x,y) \le \frac{L}{2}\|x-y\|^2,\qquad \forall x,y\in\mathbb{R}^d.
\end{equation}

The rest of the paper is organized as follows. Section~\ref{sec:flowanalysis} analyzes the continuous-time Adam-HNAG flow. Section~\ref{sec:Adam-HNAG} presents the semi-implicit discrete {\it Adam-HNAG} method and establishes its convergence. Section~\ref{sec:adamhnags} introduces the synchronous variant {\it Adam-HNAG-s} and proves its convergence. Section~\ref{sec:experiments} reports numerical results, and Section~\ref{sec:conclusion} concludes the paper.

%%%%%%%%%%%%%%%%%%%%%%%%%%%%%%%%
\section{Flow and Energy Decay}\label{sec:flowanalysis}
%%%%%%%%%%%%%%%%%%%%%%%%%%%%%%%%

This section introduces the Adam-HNAG flow and proves exponential decay of a nonnegative energy. To simplify the notation, we suppress the time variable throughout.

\subsection{Adam-HNAG flow}
We start from the continuous-time Adam model
\begin{equation}\label{eq:adam_flow_revised}
	\left\{
	\begin{aligned}
		x' &= -\frac{m}{\sqrt{V}},\\
		\tau_1 m' &= -m + \nabla f(x), \\
		\tau_2 V' &= -V + G^2,
	\end{aligned}
	\right.
\end{equation}
where $\tau_1,\tau_2>0$ are time-rescaling constants and $G^2=\mathrm{diag}\big(\nabla f(x).^2\big)$. A suitable discretization of \eqref{eq:adam_flow_revised} leads to the full-batch Adam iteration \eqref{eq:adam}. To expose a tractable Lyapunov structure, we apply a variable-and-operator-splitting (VOS) construction~\cite{chen2025acceleratedgradientmethodsvariable}. Writing $P=\sqrt{V}$ and splitting the momentum as
$$
\frac{m}{\sqrt{V}}=x-y,
$$
we obtain, for $\tau_1=1$ and $\tau_2=\tfrac12$, the approximate split system
\begin{equation}\label{eq:vosadamflow}
	\left\{
	\begin{aligned}
		x' &= y-x,\\
		y' &= -P^{-1}\nabla f(x), \\
		P' &= -P+P^{-1}G^2.
	\end{aligned}
	\right.
\end{equation}
In this derivation, we drop the transport term $P^{-1}P'(x-y)$. Thus, \eqref{eq:vosadamflow} is an Adam-inspired approximate split system, not an equivalent reformulation of \eqref{eq:adam_flow_revised}. We then add a gradient correction inspired by the Hessian-driven Nesterov accelerated gradient (HNAG) flow~\cite{chen2019orderoptimizationmethodsbased}. Hessian-driven damping has been studied in earlier work~\cite{alvarez2002hessian,attouch2020firstorder}; here we use the first-order HNAG representation in~\cite{chen2019orderoptimizationmethodsbased}. The resulting system, called the \emph{Adam-HNAG flow}, is given in \eqref{eq:Adaptive-HNAG}. We show that it admits an energy with exponential decay.

We consider the general Adam-HNAG flow
\begin{equation}\label{eq:Adaptive-HNAG}
	\left\{
	\begin{aligned}
		x' &= y-x-\beta P^{-1}\nabla f(x), \\
		y' &= -P^{-1}\nabla f(x), \\
		P' &= -P+\gamma P^{-1}G^2,
	\end{aligned}
	\right.
\end{equation}
with initial conditions $x(0)=x_0$, $y(0)=y_0$, and $P(0)=P_0$, where $P_0$ is diagonal and positive definite. Here $\beta:[0,\infty)\to(0,\infty)$ is continuously differentiable, $\gamma:[0,\infty)\to[0,\infty)$ is a time-scaling factor, and $G^2=\mathrm{diag}\big(\nabla f(x).^2\big)$. If $P_0$ is diagonal, the $P$-equation preserves the diagonal structure of $P$, since both $P^{-1}$ and $G^2$ are diagonal.

When $P$ is scalar and $\gamma=0$, \eqref{eq:Adaptive-HNAG} has the HNAG form~\cite{chen2019orderoptimizationmethodsbased}. Its relation to Adam follows from \eqref{eq:adam_flow_revised}--\eqref{eq:vosadamflow}: the approximate split system corresponds to $\beta=0$ and $\gamma=1$. The $P$-equation is essential to the Adam-type structure. It introduces feedback control through the gradient magnitude.

\subsection{Energy decay}

%The strong Lyapunov inequality used below has the form~\cite{chen2021unifiedconvergenceanalysisorder}
%\begin{equation}\label{strong-Lyapunov}
%	\frac{\mathrm{d}}{\mathrm{d}t}\mathcal{E}(u(t)) \le -c_L\,\mathcal{E}(u(t)).
%\end{equation}

Assume that \(\arg\min f\neq\varnothing\), fix \(x^\star\in\arg\min f\), and let $z=(x,y)^\top\in\mathbb{R}^{2d}$. We define
\begin{equation}\label{lyapunov}
	\mathcal{E}(z,P):=f(x)-f(x^\star)+\frac{1}{2}\|y-x^\star\|_P^2.
\end{equation}
Because \(P(t)\) may decay and \(f\) need not have a unique minimizer, \(\mathcal E\) need not control the full state uniformly. The result below is therefore an energy-decay statement, not exponential stability of \((x,y,P)\).

\begin{theorem}[Exponential energy decay]\label{theorem1}
	Let $f\in\mathcal{S}_L$ with \(\arg\min f\neq\varnothing\), fix \(x^\star\in\arg\min f\), and let $z(0)=z_0$ and $P(0)=P_0$, where \(P_0\) is diagonal and \(P_0\succ 0\). Define $\mathcal{E}$ by \eqref{lyapunov}. Assume that a solution of \eqref{eq:Adaptive-HNAG} exists with \(P(t)\succ0\) and, for all $t\ge 0$,
\begin{equation}\label{eq:betagamma}
	\gamma(t)\|y(t)-x^\star\|_\infty^2\le 2\beta(t).
\end{equation}
	Then every solution of \eqref{eq:Adaptive-HNAG} satisfies
	\begin{equation}\label{eq:expdecay}
		\mathcal{E}(z(t),P(t)) \le \mathcal{E}(z_0,P_0)e^{-t}.
	\end{equation}
\end{theorem}

\begin{proof}
	Differentiating \eqref{lyapunov} along \eqref{eq:Adaptive-HNAG} yields
	\begin{equation}\label{eq:Edot-start-new}
		\begin{aligned}
			&\quad \frac{\mathrm{d}}{\mathrm{d}t}\mathcal{E}(z(t),P(t))\\
			&= \langle\nabla f(x),x'\rangle + \big\langle P(y-x^\star),y'\big\rangle + \frac{1}{2}\|y-x^\star\|_{P'}^2 \\
			&= \langle\nabla f(x),y-x-\beta P^{-1}\nabla f(x)\rangle
			- \langle P(y-x^\star),P^{-1}\nabla f(x)\rangle\\
			&\quad + \frac{1}{2}\|y-x^\star\|_{-P+\gamma P^{-1}G^2}^2 \\
			&= - \langle\nabla f(x), x - x^\star\rangle
			-\frac{1}{2}\|y-x^\star\|_P^2
			-\beta\|\nabla f(x)\|_{P^{-1}}^2
			+\frac{\gamma}{2}\|y-x^\star\|_{P^{-1}G^2}^2.
		\end{aligned}
	\end{equation}
	By convexity,
	\[
	\langle\nabla f(x),x - x^\star\rangle \ge f(x)-f(x^\star).
	\]
	Substituting into \eqref{eq:Edot-start-new} gives
	\begin{equation}\label{eq:Edot-mid-new}
		\frac{\mathrm{d}}{\mathrm{d}t}\mathcal{E}(z(t),P(t))
		\le -\mathcal{E}(z(t),P(t))
		-\beta\|\nabla f(x)\|_{P^{-1}}^2
		+\frac{\gamma}{2}\|y-x^\star\|_{P^{-1}G^2}^2.
	\end{equation}
	Since $P$ and $G^2$ are diagonal,
	\begin{equation}\label{eq:force_bound}
		\|y-x^\star\|_{P^{-1}G^2}^2
		=\sum_{i=1}^d p_i^{-1}(y_i-x_i^\star)^2(\partial_i f(x))^2
		\le \|y-x^\star\|_\infty^2\,\|\nabla f(x)\|_{P^{-1}}^2.
	\end{equation}
	Combining \eqref{eq:Edot-mid-new}, \eqref{eq:force_bound}, and the assumption \eqref{eq:betagamma} yields
	\[
	\frac{\mathrm{d}}{\mathrm{d}t}\mathcal{E}(z(t),P(t))\le -\mathcal{E}(z(t),P(t)).
	\]
	Integrating proves \eqref{eq:expdecay}.
\end{proof}

The correction term in the $x$-equation introduces the negative contribution $\|\nabla f(x)\|_{P^{-1}}^2 = {\rm Trace}(P^{-1}G^2)$. At the same time, the adaptive metric generates the positive term $\|y-x^\star\|_{P^{-1}G^2}^2$. Since both $P$ and $G^2$ are diagonal, these terms have the same structure. Choosing $\beta$ and $\gamma$ in a compatible way makes the positive term controlled by the negative one. This cancellation is the key to the exponential energy estimate.

%%%%%%%%%%%%%%%%%%%%%%%%%%%%%%%%
\section{Adam-HNAG: Scheme and Convergence Analysis}\label{sec:Adam-HNAG}
%%%%%%%%%%%%%%%%%%%%%%%%%%%%%%%%

This section introduces a discretization of the Adam-HNAG flow and studies its convergence. The Lyapunov analysis yields a discrete contraction estimate and a corresponding practical parameter choice.

\subsection{Scheme}

Consider the following implicit--explicit (IMEX) discretization of \eqref{eq:Adaptive-HNAG}:
\begin{subequations}\label{eq:AdamHNAG}
	\begin{align}
		\frac{x_{k+1}-x_k}{\alpha_k} &= y_k-x_{k+1}-\beta_k P_{k-1}^{-1}\nabla f(x_k), \label{AHNAG:a}\\
		\frac{y_{k+1}-y_k}{\alpha_k} &= -P_k^{-1}\nabla f(x_{k+1}), \label{AHNAG:b}\\
		\frac{P_{k+1}-P_k}{\alpha_k} &= -P_{k+1}+\gamma_k P_k^{-1}G^2_{k+1}, \label{AHNAG:c}
	\end{align}
\end{subequations}
with initial values $x_0$, $y_0$, and diagonal $P_0\succ 0$.
For the lagged metric at \(k=0\), set \(P_{-1}:=P_0\).
Here $\alpha_k>0$ is the step size, $\beta_k,\gamma_k>0$ are parameters, and
$G_{k+1}^2=\mathrm{diag}\big(\nabla f(x_{k+1}).^2\big).$ The parameters $(\alpha_k,\beta_k,\gamma_k)$ are chosen to guarantee convergence. We combine $\alpha_k$ and $\beta_k$ into the single parameter $\eta_k:=\alpha_k\beta_k$.

When $P_k$ is scalar and $\gamma_k=0$, the method reduces to HNAG; see~\cite{chen2021unifiedconvergenceanalysisorder} for the deterministic setting and~\cite{yu2026shangrobuststochasticacceleration} for the stochastic setting. Thus, \eqref{eq:AdamHNAG} can be viewed as an Adam-style extension of HNAG. The feedback through $G_{k+1}^2$ drives the adaptive preconditioner and can tighten the product Lyapunov bound.

Algorithm~\ref{alg:adam-hnag} summarizes the implementable method. It replaces $P_k^{-1}$ by $(P_k+\varepsilon I)^{-1}$, where $\varepsilon\ge0$ is a stabilization constant, and sets $\gamma_k=\alpha_k/R^2$. Here $R>0$ controls the preconditioner update: a larger $R$ gives a smaller $\gamma_k$.

In Theorem~\ref{thm:lyapunov-contraction}, $R$ is an a priori bound on $\|y_k-x^\star\|_\infty$, and $\varepsilon=0$. For $\varepsilon>0$, one may use the Lyapunov metric $P_k+\varepsilon I$. The same argument applies, except that the metric-variation estimate contains
$$
-\frac{\alpha_k}{2}\|y_{k+1}-x^\star\|_{P_{k+1}}^2
=
-\frac{\alpha_k}{2}\|y_{k+1}-x^\star\|_{P_{k+1}+\varepsilon I}^2
+\frac{\alpha_k\varepsilon}{2}\|y_{k+1}-x^\star\|^2.
$$
The trajectory bound gives $\|y_{k+1}-x^\star\|^2\le dR^2$, so the additional one-step term is at most $\alpha_k\varepsilon dR^2/2$. Thus, stabilization introduces an explicit additive perturbation and, for controlled step sizes, an error level of order $\varepsilon dR^2$.
\subsection{An auxiliary estimate for preconditioned gradient descent}

For {\it Adam-HNAG}, we use the lagged preconditioner and define one preconditioned gradient step
\begin{equation}\label{eq:xkplus-lagged}
	x_k^+:=x_k-\eta_kP_{k-1}^{-1}\nabla f(x_k),\qquad k\ge0,
\end{equation}
with the convention $P_{-1}=P_0$. Algorithm~\ref{alg:adam-hnag} returns $x_k^+$ because the convergence rate is proved for $x_k^+$ rather than $x_k$.

 %\LC{The algorithm freezes $\eta_k$ once $x_k$ is accepted and backtracks only $\alpha_k$. Changing $\eta_k$ afterward would change $x_k^+$ retroactively and break the telescoping Lyapunov argument.}

\begingroup
\renewcommand{\baselinestretch}{1.2}\selectfont
\begin{algorithm}[H]
	\caption{Adam-HNAG}
	\label{alg:adam-hnag}
	\begin{algorithmic}[1]
		\REQUIRE Initial data \(x_0,y_0\in\mathbb{R}^d\), diagonal \(P_0\succ0\), \(L\), \(R>0\), \(T\ge1\), \(\mathrm{InnerLoop}\in\{\mathrm{True},\mathrm{False}\}\), and \(\varepsilon\ge0\)
		\STATE Set \(P_{-1}=P_0\), \(g_0=\nabla f(x_0)\); if \(g_0=0\), return \(x_0\)
		\STATE \(\eta_0=\dfrac{1}{L}\dfrac{\|g_0\|_{(P_{-1}+\varepsilon I)^{-1}}^2}{\|g_0\|_{(P_{-1}+\varepsilon I)^{-2}}^2}\)
		\FOR{\(k=0,1,\ldots,T-1\)}
		\STATE Set \(\alpha_k=\sqrt{\eta_k/2}\) and \(\gamma_k=\alpha_k/R^2\)
		\STATE Set \(\mathrm{Accept}=\mathrm{False}\)
		\WHILE{\(\mathrm{Accept}=\mathrm{False}\)}
		\STATE \(x_{k+1}^{\rm tr}=\dfrac{1}{1+\alpha_k}x_k+\dfrac{\alpha_k}{1+\alpha_k}y_k-\dfrac{\eta_k}{1+\alpha_k}(P_{k-1}+\varepsilon I)^{-1}g_k\)
		\STATE Set \(g_{k+1}^{\rm tr}=\nabla f(x_{k+1}^{\rm tr})\); if \(g_{k+1}^{\rm tr}=0\), return \(x_{k+1}^{\rm tr}\)
		\STATE \(y_{k+1}^{\rm tr}=y_k-\alpha_k(P_k+\varepsilon I)^{-1}g_{k+1}^{\rm tr}\)
		\STATE \((G_{k+1}^{2})^{\rm tr}=\operatorname{diag}((g_{k+1}^{\rm tr}).^2)\)
		\STATE \(P_{k+1}^{\rm tr}=\dfrac{1}{1+\alpha_k}P_k+\dfrac{\alpha_k\gamma_k}{1+\alpha_k}(P_k+\varepsilon I)^{-1}(G_{k+1}^{2})^{\rm tr}\)
		\STATE \(\eta_{k+1}^{\rm tr}=\dfrac{1}{L}\dfrac{\|g_{k+1}^{\rm tr}\|_{(P_k+\varepsilon I)^{-1}}^2}{\|g_{k+1}^{\rm tr}\|_{(P_k+\varepsilon I)^{-2}}^2}\)
		\IF{\(\mathrm{InnerLoop}=\mathrm{False}\) or \(2\alpha_k^2\le \eta_{k+1}^{\rm tr}(1+\alpha_k)\)}
		\STATE Set \(x_{k+1}=x_{k+1}^{\rm tr}\), \(y_{k+1}=y_{k+1}^{\rm tr}\), \(P_{k+1}=P_{k+1}^{\rm tr}\), \(g_{k+1}=g_{k+1}^{\rm tr}\), and \(\eta_{k+1}=\eta_{k+1}^{\rm tr}\)
		\STATE Set \(\mathrm{Accept}=\mathrm{True}\)
		\ELSE
		\STATE Set \(\alpha_k=\sqrt{\eta_{k+1}^{\rm tr}/2}\) and \(\gamma_k=\alpha_k/R^2\)
		\ENDIF
		\ENDWHILE
		\ENDFOR
		\STATE \(x_T^+=x_T-\eta_T(P_{T-1}+\varepsilon I)^{-1}g_T\)
		\RETURN the postprocessed point \(x_T^+\)
	\end{algorithmic}
\end{algorithm}
\endgroup

With this notation, the scheme becomes a sequence of preconditioned gradient updates. Given $(x_k,x_k^+,y_k)$ and $P_k$, compute
\begin{subequations}\label{eq:AdamHNAGplus}
	\begin{align}
		\frac{x_{k+1}-x_k^+}{\alpha_k} &= y_k-x_{k+1}, \label{AHNAG2:a}\\
		\frac{y_{k+1}-y_k}{\alpha_k} &= -P_k^{-1}\nabla f(x_{k+1}), \label{AHNAG2:b}\\
		x_{k+1}^+ &= x_{k+1}-\eta_{k+1}P_k^{-1}\nabla f(x_{k+1}), \label{AHNAG2:c}\\
		\frac{P_{k+1}-P_k}{\alpha_k} &= -P_{k+1}+\gamma_kP_k^{-1}G_{k+1}^2. \label{AHNAG2:d}
	\end{align}
\end{subequations}

We first establish a sufficient-decrease estimate for one preconditioned gradient step.

\begin{lemma}[Directional descent in $\|\cdot\|_{P^{-1}}$]\label{lemma1-euclid-Bform}
	Assume that $f$ is $L$-smooth in the Euclidean norm and that $P^{-1}\succ0$ is symmetric. Suppose also that $\nabla f(x)\neq0$. Define $x^+$ by
\begin{equation}\label{auxi1}
	x^+ := x-\eta P^{-1}\nabla f(x),
\end{equation}
	 and set
	\begin{equation}\label{eq:etabar}
		\bar\eta(P^{-1},\nabla f(x))
		:=
		\frac{1}{L}\frac{\|\nabla f(x)\|_{P^{-1}}^2}{\|\nabla f(x)\|_{P^{-2}}^2}.
	\end{equation}
	If $0<\eta\le\bar\eta(P^{-1},\nabla f(x))$, then
	$$
	f(x^+)\le f(x)-\frac{\eta}{2}\|\nabla f(x)\|_{P^{-1}}^2.
	$$
\end{lemma}

\begin{proof}
	By $L$-smoothness,
	$$
	f(x^+)\le f(x)+\langle\nabla f(x),x^+-x\rangle+\frac{L}{2}\|x^+-x\|^2.
	$$
	Since $x^+-x=-\eta P^{-1}\nabla f(x)$,
	$$
	\langle\nabla f(x),x^+-x\rangle=-\eta\|\nabla f(x)\|_{P^{-1}}^2,
	\qquad
	\|x^+-x\|^2=\eta^2\|\nabla f(x)\|_{P^{-2}}^2.
	$$
	Therefore,
	$$
	f(x^+)\le f(x)-\eta\left(1-\frac{L\eta}{2}\frac{\|\nabla f(x)\|_{P^{-2}}^2}{\|\nabla f(x)\|_{P^{-1}}^2}\right)\|\nabla f(x)\|_{P^{-1}}^2.
	$$
	The step-size condition makes the factor in parentheses at least $1/2$, which proves the result.
\end{proof}

If $\nabla f(x)=0$, then $\bar\eta(P^{-1},\nabla f(x))$ is not used. In this case, $x^+=x$, the descent estimate holds with equality for every $\eta>0$, and all gradient-dependent terms in the subsequent Lyapunov estimates vanish.

\subsection{A discrete strong Lyapunov inequality}

Define
\begin{equation}\label{lyapunov2}
	\mathcal{E}(z_k^+,P_k):=f(x_k^+)-f(x^\star)+\frac{1}{2}\|y_k-x^\star\|_{P_k}^2,
\end{equation}
where $z_k^+=(x_k^+,y_k)^\top$. We decompose the energy difference as
\begin{equation}\label{eq:decomp}
	\mathcal{E}(z_{k+1}^+,P_{k+1})-\mathcal{E}(z_k^+,P_k)=\mathbb{I}_1+\mathbb{I}_2+\mathbb{I}_3,
\end{equation}
where
$$
	\begin{aligned}
		\mathbb{I}_1 &:= \mathcal{E}(z_{k+1}^+,P_{k+1})-\mathcal{E}(z_{k+1},P_{k+1}),\\
		\mathbb{I}_2 &:= \mathcal{E}(z_{k+1},P_{k+1})-\mathcal{E}(z_{k+1},P_k),\\
		\mathbb{I}_3 &:= \mathcal{E}(z_{k+1},P_k)-\mathcal{E}(z_k^+,P_k),
	\end{aligned}
$$
and $z_k=(x_k,y_k)^\top$. Here $\mathbb{I}_1$ measures the decrease from the preconditioned gradient step, $\mathbb{I}_2$ measures the change in the metric, and $\mathbb{I}_3$ measures the change from the HNAG update. Estimating $\mathbb{I}_3$ is the main technical step.

If $\nabla f(x_{k+1})=0$, the estimate below holds by the zero-gradient convention. Otherwise, Lemma~\ref{lemma1-euclid-Bform}, applied at $x_{k+1}$ under the condition
$$
0<\eta_{k+1}\le\bar\eta(P_k^{-1},\nabla f(x_{k+1})),
$$
gives
\begin{equation}\label{equ1}
	\mathbb{I}_1=f(x_{k+1}^+)-f(x_{k+1})\le-\frac{\eta_{k+1}}{2}\|\nabla f(x_{k+1})\|_{P_k^{-1}}^2.
\end{equation}
Using \eqref{AHNAG2:d}, we obtain
\begin{equation}\label{equ2}
	\mathbb{I}_2
	=
	\frac{1}{2}\|y_{k+1}-x^\star\|_{P_{k+1}-P_k}^2
	=
	-\frac{\alpha_k}{2}\|y_{k+1}-x^\star\|_{P_{k+1}}^2
	+\frac{\alpha_k\gamma_k}{2}\|y_{k+1}-x^\star\|_{P_k^{-1}G^2_{k+1}}^2.
\end{equation}

\begin{lemma}\label{lem:I3}
	Let $f\in\mathcal{S}_L$. Given $x_0$, $y_0$, and a diagonal matrix $P_0\succ0$, generate $(x_k,y_k,P_k)$ by \eqref{eq:AdamHNAG} and define $x_k^+$ by \eqref{eq:xkplus-lagged}. Then
	$$
	\mathbb{I}_3
	\le
	-\alpha_k\big(f(x_{k+1})-f(x^\star)\big)
	+\frac{\alpha_k^2}{2}\|\nabla f(x_{k+1})\|_{P_k^{-1}}^2.
	$$
\end{lemma}

\begin{proof}
	Using \eqref{eq:AdamHNAG}, \eqref{eq:AdamHNAGplus}, and the definition of $\mathcal{E}$, we obtain
	\begin{equation}\label{equ4}
		\begin{aligned}
			\mathcal{E}(z_{k+1},P_k)-\mathcal{E}(z_k^+,P_k)
			&=
			\langle\nabla_z\mathcal{E}(z_{k+1},P_k),z_{k+1}-z_k^+\rangle
			-D_{\mathcal{E}}(z_k^+,z_{k+1};P_k) \\
			&=
			\alpha_k\langle\nabla f(x_{k+1}),y_k-x_{k+1}\rangle
			-D_{\mathcal{E}}(z_k^+,z_{k+1};P_k) \\
			&\quad
			+\alpha_k\langle P_k(y_{k+1}-x^\star),-P_k^{-1}\nabla f(x_{k+1})\rangle \\
			&=
			-\alpha_k\langle\nabla f(x_{k+1}),x_{k+1}-x^\star\rangle \\
			&\quad
			+\alpha_k\langle\nabla f(x_{k+1}),y_k-y_{k+1}\rangle
			-D_{\mathcal{E}}(z_k^+,z_{k+1};P_k).
		\end{aligned}
	\end{equation}
	Convexity gives
	\begin{equation}\label{equ5}
		-\langle\nabla f(x_{k+1}),x_{k+1}-x^\star\rangle
		\le
		-\big(f(x_{k+1})-f(x^\star)\big).
	\end{equation}
	Moreover,
	\begin{equation}\label{equ8}
		-D_{\mathcal{E}}(z_k^+,z_{k+1};P_k)
		\le
		-\frac{1}{2}\|y_k-y_{k+1}\|_{P_k}^2.
	\end{equation}
	Young's inequality in the $P_k$-metric yields
	\begin{equation}\label{equ6}
		\alpha_k\langle\nabla f(x_{k+1}),y_k-y_{k+1}\rangle
		\le
		\frac{\alpha_k^2}{2}\|\nabla f(x_{k+1})\|_{P_k^{-1}}^2
		+\frac{1}{2}\|y_k-y_{k+1}\|_{P_k}^2.
	\end{equation}
	Substituting \eqref{equ5}, \eqref{equ6}, and \eqref{equ8} into \eqref{equ4} proves the result.
\end{proof}

Combining the bounds on $\mathbb{I}_1$, $\mathbb{I}_2$, and $\mathbb{I}_3$, we obtain the following discrete Lyapunov inequality.

\begin{lemma}\label{lemma3}
	Let $f\in\mathcal{S}_L$. Starting from $x_0$, $y_0$, and diagonal $P_0\succ0$, generate $(x_k,y_k,P_k)$ by \eqref{eq:AdamHNAG} and define $x_k^+$ by \eqref{eq:xkplus-lagged}. %Denoted by $L_{k+1} = 1/\eta_{k+1}$. 
	Assume
	\[
	\sup_{k\ge0}\|y_k-x^\star\|_\infty\le R.
	\]
  If \(\nabla f(x_{k+1})\neq0\), assume in addition that
	\[
	0<\eta_{k+1} \le \bar\eta(P_k^{-1},\nabla f(x_{k+1})).
	\]
	Then
	\begin{equation}\label{eq:strong-lyap-discrete}
		\mathcal{E}(z_{k+1}^+,P_{k+1})-\mathcal{E}(z_k^+,P_k)
		\le
		-\alpha_k \mathcal{E}(z_{k+1}^+,P_{k+1})
		+\frac{1}{2}\|\nabla f(x_{k+1})\|_{M_{k+1}}^2,
	\end{equation}
	where
	\[
	M_{k+1}
	:=
	\big(\alpha_k^2+\alpha_k\gamma_k R^2-\eta_{k+1}(1+\alpha_k)\big)P_k^{-1}.
	\]
\end{lemma}

\begin{proof}
	Summing the bounds for $\mathbb{I}_1$, $\mathbb{I}_2$, and $\mathbb{I}_3$ gives
	\[
	\begin{aligned}
		\mathcal{E}(z_{k+1}^+,P_{k+1})-\mathcal{E}(z_k^+,P_k)
		&\le
		-\alpha_k\mathcal{E}(z_{k+1},P_{k+1})
		+\frac{\alpha_k^2-\eta_{k+1}}{2}\|\nabla f(x_{k+1})\|_{P_k^{-1}}^2 \\
		&\quad
		+\frac{\alpha_k\gamma_k}{2}\|y_{k+1}-x^\star\|_{P_k^{-1}G^2_{k+1}}^2.
	\end{aligned}
	\]
	Using \(\|y_{k+1}-x^\star\|_\infty\le R\) and \(G^2_{k+1}=\mathrm{diag}(\nabla f(x_{k+1}).^2)\),
	\[
	\frac{\alpha_k\gamma_k}{2}\|y_{k+1}-x^\star\|_{P_k^{-1}G^2_{k+1}}^2
	\le
	\frac{\alpha_k\gamma_kR^2}{2}\|\nabla f(x_{k+1})\|_{P_k^{-1}}^2.
	\]
	Next, if \(\nabla f(x_{k+1})=0\), the following inequality is an equality; otherwise it follows from Lemma~\ref{lemma1-euclid-Bform}:
	\[
	-\alpha_k\mathcal{E}(z_{k+1},P_{k+1})
	\le
	-\alpha_k\mathcal{E}(z_{k+1}^+,P_{k+1})
	-\frac{\alpha_k\eta_{k+1}}{2}\|\nabla f(x_{k+1})\|_{P_k^{-1}}^2.
	\]
	Substituting these bounds yields \eqref{eq:strong-lyap-discrete}.
\end{proof}

\subsection{Parameter choice and Lyapunov contraction}\label{sec:alphachoice}

Lemma~\ref{lemma3} reduces the one-step decay to a single quadratic correction term. Recall that
\[
M_{k+1}
=
\big(\alpha_k^2+\alpha_k\gamma_k R^2-\eta_{k+1}(1+\alpha_k)\big)P_k^{-1},
\qquad
\eta_{k+1}:=\alpha_{k+1}\beta_{k+1}.
\]
A sufficient condition for cancellation is therefore
\begin{equation}\label{eq:Mnegative}
	M_{k+1}\preceq 0
	\qquad\Longleftrightarrow\qquad
	\alpha_k^2+\alpha_k\gamma_k R^2-\eta_{k+1}(1+\alpha_k)\le 0.
\end{equation}

We first state the convergence result under an assumption, and then discuss how to enforce it in practice.

\begin{theorem}[Lyapunov contraction]\label{thm:lyapunov-contraction}
	Let \(f\in\mathcal S_L\) with \(\arg\min f\neq\varnothing\), and fix \(x^\star\in\arg\min f\). Starting from \(x_0,y_0\), and diagonal \(P_0\succ0\), set \(P_{-1}=P_0\) and generate \((x_k,y_k,P_k)\) by the exact scheme \eqref{eq:AdamHNAG}. Define \(x_k^+\) by \eqref{eq:xkplus-lagged} and \(\mathcal E\) by \eqref{lyapunov2}. Assume that there exists \(R>0\) such that
	\begin{equation}\label{eq:Ybound}
		\sup_{k\ge0}\|y_k-x^\star\|_\infty\le R.
	\end{equation}
	Define $\eta_k = \alpha_k\beta_k$. For each nonstationary iterate, choose
	\[
	\eta_k=\bar\eta(P_{k-1}^{-1},\nabla f(x_k)),
	\qquad
	0<\alpha_k\le\sqrt{\eta_k/2},
	\qquad
	\gamma_k=\alpha_k/R^2.
	\]
	Assume further that, for all \(k\ge0\),
	\begin{equation}\label{eq:alpha-consistency}
		2\alpha_k^2\le \eta_{k+1}(1+\alpha_k).
	\end{equation}
	Then
	\begin{equation}\label{eq:Ek-contraction}
		\mathcal{E}(z_{k+1}^+,P_{k+1})
		\le
		\mathcal{E}(z_0^+,P_0)\prod_{j=0}^{k}\frac{1}{1+\alpha_j},
		\qquad \forall k\ge0.
	\end{equation}
\end{theorem}

\begin{proof}
	By Lemma~\ref{lemma3}, we have
	\[
	\mathcal{E}(z_{k+1}^+,P_{k+1})-\mathcal{E}(z_k^+,P_k)
	\le
	-\alpha_k \mathcal{E}(z_{k+1}^+,P_{k+1})
	+\frac{1}{2}\|\nabla f(x_{k+1})\|_{M_{k+1}}^2.
	\]
	Using \(\gamma_kR^2=\alpha_k\), we obtain
	\[
	M_{k+1}
	=
	\big(2\alpha_k^2-\eta_{k+1}(1+\alpha_k)\big)P_k^{-1}.
	\]
	Assumption \eqref{eq:alpha-consistency} therefore implies \(M_{k+1}\preceq0\), and thus
	\[
	\mathcal{E}(z_{k+1}^+,P_{k+1})-\mathcal{E}(z_k^+,P_k)
	\le
	-\alpha_k \mathcal{E}(z_{k+1}^+,P_{k+1}).
	\]
	Rearranging gives
	\[
	\mathcal{E}(z_{k+1}^+,P_{k+1})
	\le
	\frac{1}{1+\alpha_k}\mathcal{E}(z_k^+,P_k).
	\]
	Iterating this inequality yields \eqref{eq:Ek-contraction}.
	If \(\nabla f(x_\ell)=0\) for some iterate, convexity gives \(x_\ell\in\arg\min f\), so the method can be terminated at that point; otherwise the adaptive stepsizes above are well defined.
\end{proof}

\begin{proposition}[Accelerated rate under exact adaptive steps]\label{prop:nonnegative-feedback-rate}
	Under the assumptions and parameter choices of Theorem~\ref{thm:lyapunov-contraction}, suppose that every nonstationary iterate accepts the exact adaptive step
	$$
	\alpha_k=\sqrt{\frac{\eta_k}{2}},
	\qquad
	\eta_k=\bar\eta(P_{k-1}^{-1},\nabla f(x_k)).
	$$
	Then
	$$
	\rho_k:=\prod_{j=0}^k(1+\alpha_j)^{-1}=\mathcal{O}(k^{-2}),
	$$
	and hence
	$$
	\mathcal{E}(z_{k+1}^+,P_{k+1})=\mathcal{O}(k^{-2}).
	$$
\end{proposition}

\begin{proof}
	Since $\gamma_k\geq0$ and $G_{k+1}^2$ is positive semidefinite, the $P$-update gives
	$$
	P_{k+1}\succeq\frac{1}{1+\alpha_k}P_k.
	$$
	Therefore,
	$$
	P_k\succeq\rho_{k-1}P_0,
	\qquad
	\rho_{k-1}:=\prod_{j=0}^{k-1}(1+\alpha_j)^{-1},
	$$
	with $\rho_{-1}=1$. Let $\mu_0:=\lambda_{\min}(P_0)>0$. For any diagonal $P\succ0$ and $g\neq0$, setting $h=P^{-1}g$ gives
	$$
	L\bar\eta(P^{-1},g)
	=
	\frac{g^\top P^{-1}g}{g^\top P^{-2}g}
	=
	\frac{h^\top Ph}{h^\top h}
	\geq
	\lambda_{\min}(P).
	$$
	For $k=0$, $P_{-1}=P_0$. For $k\geq1$,
	$$
	P_{k-1}\succeq\rho_{k-2}P_0\succeq\rho_{k-1}P_0.
	$$
	Hence
	$$
	\alpha_k
	\geq
	c_0\sqrt{\rho_{k-1}},
	\qquad
	c_0:=\sqrt{\frac{\mu_0}{2L}}.
	$$
	Define $s_k:=\rho_{k-1}^{-1/2}$. Since $\rho_{k-1}\leq1$, we have $s_k\geq1$, and
	$$
	s_{k+1}
	=
	s_k\sqrt{1+\alpha_k}.
	$$
	Therefore,
	$$
	\begin{aligned}
	s_{k+1}-s_k
	&=
	s_k\big(\sqrt{1+\alpha_k}-1\big)\geq
	s_k\left(\sqrt{1+\frac{c_0}{s_k}}-1\right)\\
	&=
	\frac{c_0}{\sqrt{1+c_0/s_k}+1}
	\geq
	\frac{c_0}{\sqrt{1+c_0}+1}
	=:\delta_0>0.
	\end{aligned}
	$$
	Thus $s_{k+1}\geq1+(k+1)\delta_0$, and consequently
	$$
	\rho_k=s_{k+1}^{-2}
	\leq
	\frac{1}{\big(1+(k+1)\delta_0\big)^2}
	=
	\mathcal{O}(k^{-2}).
	$$
	The Lyapunov contraction then gives
	$$
	\mathcal{E}(z_{k+1}^+,P_{k+1})
	\leq
	\rho_k\mathcal{E}(z_0^+,P_0)
	=
	\mathcal{O}(k^{-2}).
	$$
\end{proof}

\begin{remark}
The conclusion requires a uniform lower bound on $\alpha_k$. If backtracking can reduce $\alpha_k$ arbitrarily below $\sqrt{\eta_k/2}$, the upper bound $\alpha_k\leq\sqrt{\eta_k/2}$ alone does not imply an $O(k^{-2})$ rate. This rate matches that established for HNAG in~\cite{chen2021unifiedconvergenceanalysisorder}.
\end{remark}

%%%%%%%%%%%%%%%%%%%%%%%%%%%%%%%%
\subsection{A one-coordinate surrogate for adaptive decay}\label{sec:decayanalysis}
%%%%%%%%%%%%%%%%%%%%%%%%%%%%%%%%

Nonnegative feedback preserves the accelerated upper bound but does not, by itself, prove a faster rate. As long as the consistency condition \eqref{eq:alpha-consistency} holds, a larger accepted $\alpha_k$ decreases $\rho_k$ and tightens the product bound.

The following discussion is heuristic. To obtain a scalar surrogate, fix a coordinate $i$ and define $p_k:=(P_k)_{ii}$. The $i$-th diagonal entry of the $P$-equation satisfies
\begin{equation}\label{eq:pk-rec}
	p_{k+1}-p_k
	=
	-\alpha_kp_{k+1}
	+\alpha_k\gamma_kp_k^{-1}\big(\partial_i f(x_{k+1})\big)^2.
\end{equation}
If $P_k=p_kI$, then \eqref{eq:etabar} gives $\bar\eta(P_k^{-1},g)=p_k/L$ for every $g\neq0$. We therefore use the scalar relation
\begin{equation}\label{eq:alpha-rec}
	\alpha_k=\sqrt{\frac{p_k}{2L}},
\end{equation}
for which $2\alpha_k^2=p_k/L$. In this scalar model, this choice is compatible with the consistency condition \eqref{eq:alpha-consistency}.

\subsubsection{Forced case ($q_{k+1}\not\equiv0$)}

Set $\gamma_k=\alpha_k/R^2$ and define $q_{k+1}:=\partial_i f(x_{k+1})/R$. Substituting $\alpha_k^2=p_k/(2L)$ into \eqref{eq:pk-rec} gives
\begin{equation}\label{eq:pk-rec2}
	p_{k+1}-p_k
	=
	-\alpha_kp_{k+1}
	+
	\frac{q_{k+1}^2}{2L}.
\end{equation}
Equation \eqref{eq:pk-rec2} shows the competition between the dissipation $\alpha_kp_{k+1}$ and the forcing $q_{k+1}^2/(2L)$.

\paragraph{Constant forcing} To isolate the effect of forcing, suppose that $q_{k+1}^2\equiv q^2>0$. Then \eqref{eq:pk-rec2} becomes
$$
p_{k+1}
=
\frac{p_k+q^2/(2L)}{1+\sqrt{p_k/(2L)}}.
$$
Its unique positive fixed point is
$$
p_\star=\left(\frac{q^4}{2L}\right)^{1/3},
\qquad
\alpha_\star=\sqrt{\frac{p_\star}{2L}}
=\left(\frac{q^2}{4L^2}\right)^{1/3}.
$$
If $p_k\to p_\star$, then $\alpha_k\to\alpha_\star>0$ and
$$
\lim_{k\to\infty}\rho_k^{1/k}=(1+\alpha_\star)^{-1}<1.
$$
Thus, persistent forcing can prevent $\alpha_k$ from vanishing and yield a geometric product bound, faster than the unforced $O(k^{-2})$ bound. This conclusion concerns only the scalar surrogate.

\paragraph{Asymptotically vanishing forcing} Constant nonzero forcing is incompatible with convergence to a stationary point. Along a convergent trajectory, $q_k\to0$, and its decay relative to $p_k$ determines the product bound. If $p_k$ varies slowly, then neglecting the left-hand side of \eqref{eq:pk-rec2} suggests the quasi-steady balance
$$
\alpha_kp_k\approx\frac{q_k^2}{2L}.
$$
Combining this balance with $\alpha_k=\sqrt{p_k/(2L)}$ gives
\begin{equation}\label{eq:quasi-steady}
	p_k\approx(2L)^{-1/3}|q_k|^{4/3},
	\qquad
	\alpha_k\approx(2L)^{-2/3}|q_k|^{2/3}.
\end{equation}
This scalar model suggests two regimes.
\begin{enumerate}
	\item If $|q_k|^2=o(k^{-3})$, then the forcing is asymptotically smaller than the unforced dissipation $\alpha_kp_k\asymp k^{-3}$. The dynamics are therefore expected to recover the unforced $k^{-2}$ scaling.
	\item If $|q_k|^2$ decays more slowly than $k^{-3}$, then the forcing remains significant. Relation \eqref{eq:quasi-steady} suggests that $\alpha_k$ decays more slowly than $k^{-1}$, so $\rho_k$ decays faster than $k^{-2}$.
\end{enumerate}
The critical case $|q_k|^2\asymp k^{-3}$ separates these regimes. The forcing and dissipation then have the same order, and their coefficients determine the leading behavior.

\subsubsection{Feedback dynamics}

The scalar recursion \eqref{eq:pk-rec} should be read together with the energy estimate
\begin{equation}\label{eq:E-contract}
	\frac{1}{2L}\|\nabla f(x_{k+1}^+)\|^2
	\le
	f(x_{k+1}^+)-f(x^\star)
	\le
	\mathcal E_{k+1}
	\le
	\rho_k\mathcal E_0.
\end{equation}
Thus, faster decay of $\rho_k$ gives a faster-decaying upper bound on the Lyapunov function and the objective gap. By $L$-smoothness, this is consistent with a smaller gradient, which in turn weakens the forcing term in \eqref{eq:pk-rec}. This suggests a feedback mechanism: a larger admissible $\alpha_k$ tightens the Lyapunov bound, the resulting decrease in the gradient weakens the forcing, and the recursion moves toward the unforced accelerated regime.

This picture suggests two phases: adaptive effects are stronger when the gradients are large, while the dynamics gradually approach a $k^{-2}$-type accelerated regime as the gradients vanish.

\subsection{Practical resolution of the consistency condition}\label{sec:practical-alpha}

We next discuss how to verify \eqref{eq:alpha-consistency}. The difficulty is that \(x_{k+1}\) itself depends on \(\alpha_k\) so that \eqref{eq:alpha-consistency} is nonlinear in $\alpha_k$. If \eqref{eq:alpha-consistency} fails, equivalently,
\[
\eta_{k+1}<\frac{2\alpha_k^2}{1+\alpha_k},
\]
we reset
\[
\alpha_k\leftarrow \sqrt{\eta_{k+1}/2},
\]
recompute \(x_{k+1}\) and the updated value of \(\eta_{k+1}\), and repeat if necessary. 

We now formally discuss this inner iteration. Let \(P\succ0\) be fixed. Assume \(f\in C^1\) and that the map \(\alpha\mapsto x(\alpha)\) is continuous. The discussion concerns nonstationary trial points; if \(\nabla f(x(\alpha))=0\), the method has reached a minimizer and can stop. Define
\[
\eta(\alpha):=\bar\eta\big(P^{-1},\nabla f(x(\alpha))\big),
\]
and, starting from an initial guess \(\alpha^{(0)}\), consider the fixed point iteration
\[
\alpha^{(m+1)}:=\Phi(\alpha^{(m)}):=\sqrt{\eta(\alpha^{(m)})/2},
\]
applied whenever
\[
2(\alpha^{(m)})^2>(1+\alpha^{(m)})\,\eta(\alpha^{(m)}).
\]

\begin{proposition}[Convergence of the stepsize inner loop]\label{proposition1}
	Either the stepsize inner loop terminates after finitely many steps, or the sequence \(\{\alpha^{(m)}\}_{m\ge0}\) is strictly decreasing and converges to a limit \(\alpha^\star\ge0\) satisfying
	\[
	\alpha^\star=\Phi(\alpha^\star),
	\qquad\text{equivalently}\qquad
	2(\alpha^\star)^2=\eta(\alpha^\star).
	\]
	In particular, it satisfies the consistency condition
	\[
	2(\alpha^\star)^2\le (1+\alpha^\star)\eta(\alpha^\star).
	\]
\end{proposition}

\begin{proof}
	If the consistency condition holds at some step, the procedure stops. Otherwise, for every \(m\),
	\[
	\eta(\alpha^{(m)})<\frac{2(\alpha^{(m)})^2}{1+\alpha^{(m)}},
	\]
	and therefore
	\[
	\alpha^{(m+1)}
	=
	\sqrt{\frac{\eta(\alpha^{(m)})}{2}}
	<
	\frac{\alpha^{(m)}}{\sqrt{1+\alpha^{(m)}}}
	<
	\alpha^{(m)}.
	\]
	Thus \(\{\alpha^{(m)}\}\) is strictly decreasing. Since \(\alpha^{(m)}\ge0\), it converges to some limit \(\alpha^\star\ge0\).
	
	By continuity of \(\alpha\mapsto x(\alpha)\), the map \(\alpha\mapsto \eta(\alpha)\) is continuous, and so is \(\Phi\). The \YY{continuity} of $\Phi$ 
	gives
	\[
	\alpha^\star=\Phi(\alpha^\star),
	\]
	hence \(2(\alpha^\star)^2=\eta(\alpha^\star)\). The last inequality is immediate since \(1+\alpha^\star\ge1\).
\end{proof}

This stepsize inner loop is mainly a theoretical device. In practice, the simpler choice
\[
\alpha_k=\sqrt{\eta_k/2}
\]
typically satisfies the consistency condition after a few initial iterations.

\subsubsection{Boundedness assumption}\label{sec:boundedness}

For the analysis, we assume \eqref{eq:Ybound}. Bounding the iterates is a common difficulty in the analysis of adaptive methods~\cite{duchi11a,gupta2018shampoo,liu2024adagradanisotropicsmoothness,an2025asgoadaptivestructuredgradient,kovalev2025,xie2025structuredpreconditionersadaptiveoptimization}. When a conservative bound on $x^\star$ is available, this assumption can be enforced by projection. Choose $R>0$ such that $\|x^\star\|_\infty\le R/2$ and define
$$
\mathcal{H}_R:=\{y\in\mathbb{R}^d:\|y\|_\infty\le R/2\}.
$$
Starting with $y_0\in\mathcal{H}_R$, replace the $y$-update by
\begin{equation}\label{eq:projected-y-update}
	\begin{aligned}
		y_{k+\frac{1}{2}} &= y_k-\alpha_kP_k^{-1}\nabla f(x_{k+1}),\\
		y_{k+1} &= \operatorname*{argmin}_{y\in\mathcal{H}_R}\|y-y_{k+\frac{1}{2}}\|_{P_k}^2.
	\end{aligned}
\end{equation}
Since $\mathcal{H}_R$ is nonempty, closed, and convex, and $P_k\succ0$, the weighted projection exists and is unique. Since $P_k$ is diagonal, it can be computed by coordinatewise clipping. Its optimality condition is
$$
\left\langle P_k\left(y_{k+1}-y_{k+\frac{1}{2}}\right),y-y_{k+1}\right\rangle\ge0
\qquad\text{for all }y\in\mathcal{H}_R.
$$
Taking $y=x^\star\in\mathcal{H}_R$ and expanding the square gives
$$
\|y_{k+1}-x^\star\|_{P_k}^2
\le
\left\|y_{k+\frac{1}{2}}-x^\star\right\|_{P_k}^2.
$$
The projection leaves $x_{k+1}$ unchanged and can only decrease the $y$-component of the Lyapunov function. The proof of Theorem~\ref{thm:lyapunov-contraction} therefore remains valid. Moreover,
$$
\|y_k-x^\star\|_\infty
\le
\|y_k\|_\infty+\|x^\star\|_\infty
\le R,
$$
so \eqref{eq:Ybound} holds by construction.

%%%%%%%%%%%%%%%%%%%%%%%%%%%%%%%%
\section{Adam-HNAG-s: the synchronous Adam-style discretization}\label{sec:adamhnags}
%%%%%%%%%%%%%%%%%%%%%%%%%%%%%%%%

To complete the Adam-motivated picture, we also introduce a synchronous discretization that is closer in form to full-batch Adam.

\subsection{Adam-HNAG-s}
Consider the following discretization of \eqref{eq:Adaptive-HNAG}:
\begin{subequations}\label{eq:synAdamHNAG}
	\begin{align}
		\frac{x_{k+1}-x_k}{\alpha_k} &= y_k-x_{k+1}-\beta_k P_k^{-1}\nabla f(x_k), \label{synAHNAG:a}\\
		\frac{P_{k+1}-P_k}{\tilde{\alpha}_k} &= -P_k+\gamma_k P_{k+1}^{-1}G^2_{k+1}, \label{synAHNAG:c}\\
		\frac{y_{k+1}-y_k}{\tilde{\alpha}_k} &= -P_{k+1}^{-1}\nabla f(x_{k+1}), \label{synAHNAG:b}
	\end{align}
\end{subequations}
with given $x_0$, $y_0$, and diagonal $P_0\succ 0$. Here
$
\tilde{\alpha}_k:=\frac{\alpha_k}{1+\alpha_k}\le \alpha_k
$
is a slightly smaller step size.
%, and
%$
%G_{k+1}^2=\mathrm{diag}\big(\nabla f(x_{k+1}).^2\big).$
Since the updated preconditioner $P_{k+1}$ is used in the gradient term at $x_{k+1}$, this scheme is referred to as {\it Adam-HNAG-s}.

The corresponding implementable recursion is given in Algorithm~\ref{alg:adam-hnag-s}. Here $R>0$ is the scale parameter. The stabilized inverse is used in the gradient update and the directional step-size ratio. The closed-form update of $P_{k+1}$ is the exact solution of the unstabilized metric equation and preserves positive definiteness when $P_0\succ0$.

\renewcommand{\baselinestretch}{1.2}\selectfont
\begin{algorithm}[H]
	\caption{Adam-HNAG-s}
	\label{alg:adam-hnag-s}
	\begin{algorithmic}[1]
		\REQUIRE Initial data \(x_0,y_0\in\mathbb{R}^d\), diagonal \(P_0\succ0\), \(L\), \(R>0\), \(T\ge1\), \(\mathrm{InnerLoop}\in\{\mathrm{True},\mathrm{False}\}\), and \(\varepsilon\ge0\)
		\STATE Set \(g_0=\nabla f(x_0)\); if \(g_0=0\), return \(x_0\)
		\STATE \(\eta_0=\dfrac{1}{L}\dfrac{\|g_0\|_{(P_0+\varepsilon I)^{-1}}^2}{\|g_0\|_{(P_0+\varepsilon I)^{-2}}^2}\)
		\FOR{\(k=0,1,\ldots,T-1\)}
		\STATE Set \(\alpha_k=\sqrt{\eta_k/2}\), \(\tilde{\alpha}_k=\dfrac{\alpha_k}{1+\alpha_k}\), and \(\gamma_k=\tilde{\alpha}_k/R^2\)
		\STATE Set \(\mathrm{Accept}=\mathrm{False}\)
		\WHILE{\(\mathrm{Accept}=\mathrm{False}\)}
		\STATE \(x_{k+1}^{\rm tr}=\dfrac{1}{1+\alpha_k}x_k+\dfrac{\alpha_k}{1+\alpha_k}y_k-\dfrac{\eta_k}{1+\alpha_k}(P_k+\varepsilon I)^{-1}g_k\)
		\STATE Set \(g_{k+1}^{\rm tr}=\nabla f(x_{k+1}^{\rm tr})\); if \(g_{k+1}^{\rm tr}=0\), return \(x_{k+1}^{\rm tr}\)
		\STATE Set \((G_{k+1}^{2})^{\rm tr}=\operatorname{diag}((g_{k+1}^{\rm tr}).^2)\)
		\STATE \(P_{k+1}^{\rm tr}=\dfrac{1-\tilde{\alpha}_k}{2}P_k+\dfrac{1}{2}\sqrt{(1-\tilde{\alpha}_k)^2P_k^2+4\tilde{\alpha}_k\gamma_k(G_{k+1}^{2})^{\rm tr}}\)
		\STATE \(y_{k+1}^{\rm tr}=y_k-\tilde{\alpha}_k(P_{k+1}^{\rm tr}+\varepsilon I)^{-1}g_{k+1}^{\rm tr}\)
		\STATE \(\eta_{k+1}^{\rm tr}=\dfrac{1}{L}\dfrac{\|g_{k+1}^{\rm tr}\|_{(P_{k+1}^{\rm tr}+\varepsilon I)^{-1}}^2}{\|g_{k+1}^{\rm tr}\|_{(P_{k+1}^{\rm tr}+\varepsilon I)^{-2}}^2}\)
		\IF{\(\mathrm{InnerLoop}=\mathrm{False}\) or \(2\tilde{\alpha}_k^2\le \eta_{k+1}^{\rm tr}\)}
		\STATE Set \(x_{k+1}=x_{k+1}^{\rm tr}\), \(y_{k+1}=y_{k+1}^{\rm tr}\), \(P_{k+1}=P_{k+1}^{\rm tr}\), \(g_{k+1}=g_{k+1}^{\rm tr}\), and \(\eta_{k+1}=\eta_{k+1}^{\rm tr}\)
		\STATE Set \(\mathrm{Accept}=\mathrm{True}\)
		\ELSE
		\STATE Set \(\alpha_k=\sqrt{\eta_{k+1}^{\rm tr}/2}\), \(\tilde{\alpha}_k=\alpha_k/(1+\alpha_k)\), and \(\gamma_k=\tilde{\alpha}_k/R^2\)
		\ENDIF
		\ENDWHILE
		\ENDFOR
		\STATE \(x_T^+=x_T-\eta_T(P_T+\varepsilon I)^{-1}g_T\)
		\RETURN the postprocessed point \(x_T^+\)
	\end{algorithmic}
\end{algorithm}

%Although the right-hand side of \eqref{synAHNAG:c} involves $P_{k+1}$, the update can still be written in closed form
%$$
%P_{k+1} = \frac{1-\tilde{\alpha}_k}{2} P_k + \frac{1}{2}\sqrt{(1-\tilde{\alpha}_k)^2 P_k^2 +4\tilde{\alpha}_k \gamma_k G_{k+1}^2}. $$

\subsubsection{Convergence analysis}

%We also define the auxiliary point
%$$
%x_{k+1}^+
%=
%x_{k+1}-\eta_{k+1}P_{k+1}^{-1}\nabla f(x_{k+1}).
%$$
For Adam-HNAG-s, define
$$
x_k^+=x_k-\eta_kP_k^{-1}\nabla f(x_k),
$$
which uses the updated metric $P_k$ rather than the lagged metric $P_{k-1}$ of Adam-HNAG. Algorithm~\ref{alg:adam-hnag-s} returns this postprocessed point.

Using the Lyapunov function \eqref{lyapunov2}, we split the one-step change as
$$
\begin{aligned}
	\mathcal{E}(z_{k+1}^+,P_{k+1})-\mathcal{E}(z_k^+,P_k)
	={}&
	\mathcal{E}(z_{k+1}^+,P_{k+1})-\mathcal{E}(z_{k+1},P_{k+1})\\
	&+
	\mathcal{E}(z_{k+1},P_{k+1})-\mathcal{E}(z_k^+,P_k).
\end{aligned}
$$
If $\nabla f(x_{k+1})=0$, the estimate below follows from the zero-gradient convention. Otherwise, Lemma~\ref{lemma1-euclid-Bform} gives, under the condition
$$
0<\eta_{k+1}\le\bar\eta(P_{k+1}^{-1},\nabla f(x_{k+1})),
$$
we have the sufficient decay
\begin{equation}\label{equ21}
	\mathcal{E}(z_{k+1}^+,P_{k+1})
	\le
	\mathcal{E}(z_{k+1},P_{k+1})
	-\frac{\eta_{k+1}}{2}\|\nabla f(x_{k+1})\|_{P_{k+1}^{-1}}^2.
\end{equation}
The next lemma bounds $\mathcal{E}(z_{k+1},P_{k+1})$ in terms of $\mathcal{E}(z_k^+,P_k)$.

\begin{lemma}\label{lemma4}
	Let $f\in\mathcal{S}_L$. Starting from $x_0$, $y_0$, and diagonal $P_0\succ 0$, generate $(x_k,y_k,P_k)$ by \eqref{eq:synAdamHNAG} and define $x_k^+=x_k-\eta_kP_k^{-1}\nabla f(x_k)$. Then
	$$
	\mathcal{E}(z_{k+1}, P_{k+1})
	\le
	(1-\tilde{\alpha}_k)\mathcal{E}(z_k^+, P_k)
	+ \frac{\tilde{\alpha}_k^2}{2}\|\nabla f(x_{k+1})\|^2_{P_{k+1}^{-1}}
	+ \frac{\tilde{\alpha}_k \gamma_k}{2}\|y_k-x^\star\|_{P_{k+1}^{-1}G_{k+1}^2}^{2}.
	$$
\end{lemma}

\begin{proof}
	From the update rule~\eqref{synAHNAG:b},
	\begin{equation}\label{equ16}
		\begin{aligned}
			\frac{1}{2}\|y_{k+1}-x^\star\|_{P_{k+1}}^{2}
			&=
			\frac{1}{2}\|y_k-x^\star-\tilde{\alpha}_k P_{k+1}^{-1}\nabla f(x_{k+1})\|_{P_{k+1}}^{2} \\
			&=
			\frac{1}{2}\|y_k-x^\star\|_{P_{k+1}}^{2}
			-\tilde{\alpha}_k \langle \nabla f(x_{k+1}), y_k-x^\star \rangle
			+\frac{\tilde{\alpha}_k^2}{2}\|\nabla f(x_{k+1})\|_{P_{k+1}^{-1}}^{2}.
		\end{aligned}
	\end{equation}
	Using the update of $P_{k+1}$ in~\eqref{synAHNAG:c}, we rewrite
	\begin{equation}\label{equ19}
		\begin{aligned}
			\frac{1}{2}\|y_k-x^\star\|_{P_{k+1}}^{2}
			&=
			\frac{1}{2}\|y_k-x^\star\|_{(1-\tilde{\alpha}_k)P_k + \tilde{\alpha}_k \gamma_k P_{k+1}^{-1} G_{k+1}^2}^{2} \\
			&=
			\frac{1-\tilde{\alpha}_k}{2}\|y_k-x^\star\|_{P_k}^{2}
			+ \frac{\tilde{\alpha}_k \gamma_k}{2}\|y_k-x^\star\|_{P_{k+1}^{-1}G_{k+1}^2}^{2}.
		\end{aligned}
	\end{equation}
	
	Next, using $\alpha_k = \frac{\tilde{\alpha}_k}{1-\tilde{\alpha}_k}$, the update~\eqref{synAHNAG:a} can be rewritten as
	\begin{equation}\label{equ17}
		\tilde{\alpha}_k(y_k - x^\star)
		=
		(1-\tilde{\alpha}_k)(x_{k+1} - x_k^+) + \tilde{\alpha}_k(x_{k+1} - x^\star).
	\end{equation}
	Substituting~\eqref{equ17} into the cross term in~\eqref{equ16} and using convexity of $f$ give
	\begin{equation}\label{equ18}
		\begin{aligned}
			-\tilde{\alpha}_k \langle \nabla f(x_{k+1}), y_k -x^\star \rangle
			&=
			-\tilde{\alpha}_k \langle \nabla f(x_{k+1}), x_{k+1} -x^\star \rangle \\
			&\quad
			-(1-\tilde{\alpha}_k) \langle \nabla f(x_{k+1}), x_{k+1} - x_k^+ \rangle \\
			&\le
			-\tilde{\alpha}_k (f(x_{k+1}) - f(x^\star))
			+ (1-\tilde{\alpha}_k) (f(x_k^+) - f(x_{k+1})) \\
			&=
			-(f(x_{k+1}) - f(x^\star))
			+ (1-\tilde{\alpha}_k) (f(x_k^+) - f(x^\star)).
		\end{aligned}
	\end{equation}
	
	Combining \eqref{equ16}, \eqref{equ19}, and \eqref{equ18}, then moving
	$-(f(x_{k+1})-f(x^\star))$ to the left-hand side and using the definition of
	$\mathcal{E}$, gives the result.
\end{proof}

We next establish convergence of Adam-HNAG-s.

\begin{theorem}[Synchronous Lyapunov contraction]\label{thm:lyapunov-contraction2}
	Let $f\in\mathcal{S}_L$ with $\arg\min f\neq\varnothing$, and fix $x^\star\in\arg\min f$. Starting from $x_0$, $y_0$, and a diagonal matrix $P_0\succ0$, generate $(x_k,y_k,P_k)$ by the exact scheme \eqref{eq:synAdamHNAG}. Define $x_k^+$ as above and $\mathcal{E}$ by \eqref{lyapunov2}. Assume that there exists $R>0$ such that
	$$
	\sup_{k\ge0}\|y_k-x^\star\|_\infty\le R.
	$$
	Set $\eta_k=\alpha_k\beta_k$. At each nonstationary iterate, choose
	$$
	\eta_k=\bar\eta(P_k^{-1},\nabla f(x_k)),
	\quad
	0<\alpha_k\le\sqrt{\eta_k/2},
	\quad
	\widetilde{\alpha}_k=\frac{\alpha_k}{1+\alpha_k},
	\quad
	\gamma_k=\frac{\widetilde{\alpha}_k}{R^2}.
	$$
	Assume further that
	\begin{equation}\label{ratio2}
		2\widetilde{\alpha}_k^2
		=
		\frac{2\alpha_k^2}{(1+\alpha_k)^2}
		\le\eta_{k+1}
		\qquad\text{for all }k\ge0.
	\end{equation}
	Then
	\begin{equation}
		\mathcal{E}(z_{k+1}^+,P_{k+1})
		\le
		\mathcal{E}(z_0^+,P_0)
		\prod_{j=0}^{k}\frac{1}{1+\alpha_j}.
	\end{equation}
\end{theorem}

\begin{proof}
	Combining \eqref{equ21} with Lemma~\ref{lemma4} gives
	\begin{equation}
		\begin{aligned}
			\mathcal{E}(z_{k+1}^+,P_{k+1})
			&\le
			(1-\widetilde{\alpha}_k)\mathcal{E}(z_k^+,P_k)
			+\frac{\widetilde{\alpha}_k^2-\eta_{k+1}}{2}
			\|\nabla f(x_{k+1})\|_{P_{k+1}^{-1}}^2\\
			&\quad
			+\frac{\widetilde{\alpha}_k\gamma_k}{2}
			\|y_k-x^\star\|_{P_{k+1}^{-1}G_{k+1}^2}^2.
		\end{aligned}
	\end{equation}
	Since $P_{k+1}^{-1}G_{k+1}^2$ is diagonal and $\|y_k-x^\star\|_\infty\le R$,
	$$
	\|y_k-x^\star\|_{P_{k+1}^{-1}G_{k+1}^2}^2
	\le
	R^2\|\nabla f(x_{k+1})\|_{P_{k+1}^{-1}}^2.
	$$
	Using $\gamma_k=\widetilde{\alpha}_k/R^2$, we obtain
	$$
	\frac{\widetilde{\alpha}_k\gamma_k}{2}
	\|y_k-x^\star\|_{P_{k+1}^{-1}G_{k+1}^2}^2
	\le
	\frac{\widetilde{\alpha}_k^2}{2}
	\|\nabla f(x_{k+1})\|_{P_{k+1}^{-1}}^2.
	$$
	The total coefficient of the gradient term is therefore
	$$
	\frac{2\widetilde{\alpha}_k^2-\eta_{k+1}}{2}\le0
	$$
	by \eqref{ratio2}. Dropping this nonpositive term yields
	$$
	\mathcal{E}(z_{k+1}^+,P_{k+1})
	\le
	(1-\widetilde{\alpha}_k)\mathcal{E}(z_k^+,P_k)
	=
	\frac{1}{1+\alpha_k}\mathcal{E}(z_k^+,P_k).
	$$
	Iterating this inequality proves the result. If $\nabla f(x_\ell)=0$ for some $\ell$, then $x_\ell$ is a minimizer and the method terminates; otherwise, all adaptive step sizes are well defined.
\end{proof}

\begin{proposition}[Accelerated rate for Adam-HNAG-s]\label{prop:synchronous-rate}
	Under the assumptions and parameter choices of Theorem~\ref{thm:lyapunov-contraction2}, suppose that every nonstationary iterate accepts the exact adaptive step
	$$
	\alpha_k=\sqrt{\frac{\eta_k}{2}},
	\qquad
	\eta_k=\bar\eta(P_k^{-1},\nabla f(x_k)).
	$$
	Then
	$$
	\rho_k:=\prod_{j=0}^k(1+\alpha_j)^{-1}=\mathcal{O}(k^{-2}),
	$$
	and hence
	$$
	\mathcal{E}(z_{k+1}^+,P_{k+1})=\mathcal{O}(k^{-2}).
	$$
\end{proposition}

\begin{proof}
	Since $1-\widetilde{\alpha}_k=(1+\alpha_k)^{-1}$, the metric update gives
	$$
	P_{k+1}
	=
	\frac{1}{1+\alpha_k}P_k
	+
	\widetilde{\alpha}_k\gamma_kP_{k+1}^{-1}G_{k+1}^2
	\succeq
	\frac{1}{1+\alpha_k}P_k.
	$$
	Thus, $P_k\succeq\rho_{k-1}P_0$. The Rayleigh-quotient estimate in the proof of Proposition~\ref{prop:nonnegative-feedback-rate} gives
	$$
	\alpha_k
	\geq
	\sqrt{\frac{\lambda_{\min}(P_0)}{2L}}\sqrt{\rho_{k-1}}.
	$$
	The remaining argument is identical to that proof and yields $\rho_k=\mathcal{O}(k^{-2})$. The conclusion follows from Theorem~\ref{thm:lyapunov-contraction2}.
\end{proof}

For {\it Adam-HNAG-s}, the stepsize correction also converges under the continuity assumptions of Proposition~\ref{proposition1}, extended to the trial metric $P_{k+1}(\alpha)$. The boundedness can be also enforced by a projection; see Section \ref{sec:boundedness}. 

\section{Numerical Experiments}\label{sec:experiments}
This section presents three numerical studies. The first is a theory-aligned quadratic verification using Algorithms~\ref{alg:adam-hnag}--\ref{alg:adam-hnag-s} with $\mathrm{InnerLoop}=\mathrm{True}$. The logistic-regression experiments compare stabilized practical variants not covered by the exact theorem. The final online experiment is an out-of-scope stress test rather than a validation of the smooth convex theory.

\subsection{Quadratic diagnostic on a two-dimensional Laplacian problem}\label{subsec:laplacian-verification}

We first examine the behavior of the two proposed methods on a controlled quadratic problem with a known minimizer. We use the two-dimensional Laplacian benchmark from~\cite{chen2026hnagacceleratedgradientmethod} and compare {\it Adam-HNAG} with {\it Adam-HNAG-s}. Since $x^\star$ is known, the boundedness condition can be checked directly along the computed trajectories.
%Comparisons with baseline methods are deferred to the logistic-regression experiments in Section~\ref{subsec:deterministic_exp}.

The matrix $A$ is the linear finite element stiffness matrix for the Dirichlet Laplacian on $\Omega=[0,1]^2$, assembled with iFEM~\cite{chen2009ifem} on a uniform triangulation with mesh size $h=1/160$. We use the quadratic in error coordinates
$$
f(x)=\tfrac12(x-x^\star)^\top A(x-x^\star),\qquad x^\star=0,
$$
with $d=25{,}281$ unknowns. Equivalently, this is the homogeneous case. The eigenvalues of $A$
$$
\lambda_{k,l}=4\left(\sin^2\frac{k\pi h}{2}+\sin^2\frac{l\pi h}{2}\right),
\qquad k,l=1,\ldots,1/h-1,
$$
give $L=\lambda_{\max}=8\cos^2(\pi h/2)$ and $\kappa(A)\approx1.04\times10^4$.

%\LC{Note: do you set $\epsilon = 0$? compute $\bar \eta$ exactly? Also the function gap should be $f(x_k^+) - f(x^\star)$ not $f(x_k) - f(x^\star)$}
We implement {\it Adam-HNAG} and {\it Adam-HNAG-s}
(Algorithms~\ref{alg:adam-hnag} and~\ref{alg:adam-hnag-s}) with
\(\mathrm{InnerLoop}=\mathrm{True}\), so that the parameter choices match the
theorem-prescribed recursions. We set
\(R=2\|y_0-x^\star\|_\infty\) and verify a posteriori that
\(\|y_k-x^\star\|_\infty\le R\) throughout the run. The initial point is chosen
as \(x_0=y_0\), with each component drawn independently from
\(\operatorname{Unif}(0,1)\), and the initial preconditioner is
\[
P_0=0.05\,\frac{\|\nabla f(x_0)\|}{\sqrt d}\,I .
\]
To match the theoretical recursions exactly, no stabilization is used in this
verification experiment, i.e., \(\varepsilon=0\).

\begin{figure}[!htbp]
	\centering
	\begin{subfigure}{0.49\textwidth}
		\centering
		\includegraphics[width=\linewidth]{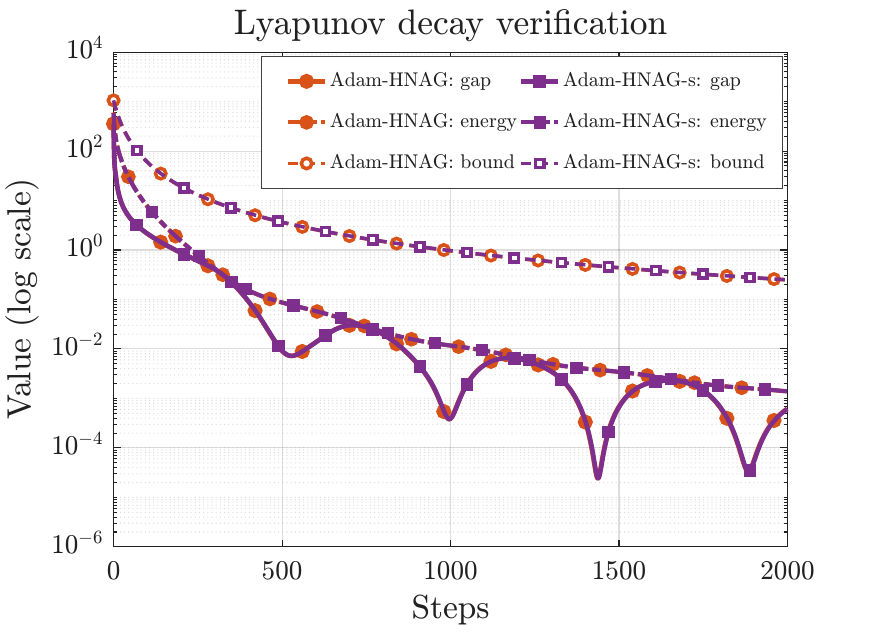}
	\end{subfigure}\hfill
	\begin{subfigure}{0.49\textwidth}
		\centering
		\includegraphics[width=\linewidth]{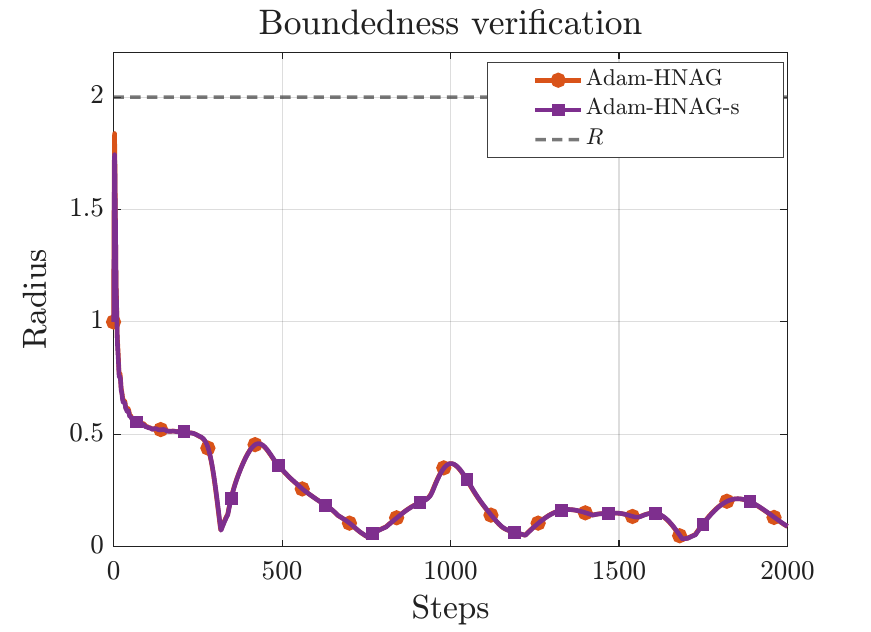}
	\end{subfigure}
	\caption{Theory-aligned verification.
		Left: decay of the auxiliary objective gap and the Lyapunov energy.
		Here ``gap'' denotes
		\(f(x_{k+1}^+)-f(x^\star)\), ``energy'' denotes
		\(\mathcal{E}(z_{k+1}^+,P_{k+1})\) \eqref{lyapunov2}, and ``bound'' denotes the
		theorem-predicted quantity
		\(
		\mathcal{E}(z_0^+,P_0)
		\prod_{i=0}^k(1+\alpha_i)^{-1}.
		\)
		The two methods produce nearly overlapping trajectories; markers
		distinguish them.
		Right: a posteriori verification of
		\(\|y_k-x^\star\|_\infty\le R\).}
	\label{fig:laplacian-theory-aligned}
\end{figure}

Figure~\ref{fig:laplacian-theory-aligned} shows that the Lyapunov energies satisfy the theorem-predicted product bounds, while the objective gaps decay alongside them and the two methods follow nearly identical trajectories. The stepsize inner loop requires no correction for {\it Adam-HNAG-s}; whenever it is triggered for {\it Adam-HNAG}, one correction suffices to accept the revised stepsize. The choice $R=2\|y_0-x^\star\|_\infty$ bounds $\|y_k-x^\star\|_\infty$ throughout this run. 
%\LC{These observations provide a diagnostic check, but they do not verify the contraction theorem because the figure uses raw iterates and a scaled objective reference rather than $x_k^+$ and the full Lyapunov bound.}

\subsection{Practical logistic-regression experiments}\label{subsec:deterministic_exp}

We consider the regularized binary logistic-regression objective
$$
\begin{aligned}
f(w,b)
&=
-\frac{1}{n}\sum_{i=1}^n
\left[
y_i\log\sigma(X_i^\top w+b)
+(1-y_i)\log\bigl(1-\sigma(X_i^\top w+b)\bigr)
\right]\\
&\quad+\frac{\lambda}{2}\|(w,b)\|^2,
\end{aligned}
$$
where $\sigma(z)=1/(1+e^{-z})$ and $\lambda=10^{-10}$ in all experiments. Since $\lambda>0$, the objective is coercive and admits a minimizer. Let
	\(X=[X_1,\ldots,X_n]^\top\in\mathbb{R}^{n\times d}\) and let
	\(\widetilde X=[X,\mathbf{1}]\) be the design matrix augmented with the
	intercept column. For all methods requiring a smoothness constant, we use the
	global bound
	\[
	L=\frac{\|\widetilde X\|_2^2}{4n}+\lambda.
	\]Writing $x=(w,b)$, we plot the regularized objective $f(x_k)$ rather than the gap $f(x_k)-f^\star$. The value $f^\star$ is unnecessary for this finite-iteration comparison because all methods are evaluated on the same objective.

\subsubsection{Methods compared.}

We compare five deterministic methods: gradient descent (GD), HNAG~\cite{chen2021unifiedconvergenceanalysisorder}, full-batch Adam \eqref{eq:adam}, {\it Adam-HNAG} (Algorithm~\ref{alg:adam-hnag}), and {\it Adam-HNAG-s} (Algorithm~\ref{alg:adam-hnag-s}). Since $x^\star$ is unknown for these logistic-regression problems, the trajectory radius cannot be verified a priori. We therefore treat $R$ as a scale parameter, set $R=2$, and use the practical variants of the proposed methods with $\mathrm{InnerLoop}=\mathrm{False}$. 

GD uses the step size $\eta=1/L$. For HNAG, we set
$$
\alpha_k=\frac{2}{k+1},
\qquad
P_k(2+\alpha_k)=\alpha_k^2L,
\qquad
\alpha_k\beta_kP_{k-1}=\frac{1}{L},
$$
which corresponds to the scalar-$P_k$, $\gamma_k=0$ case of \eqref{eq:AdamHNAG}. For full-batch Adam, we use $\beta_1=0.9$ and $\beta_2=0.999$. Since no theoretically optimal learning rate is available in this convex setting, we test
$$
\eta\in\{10^{-4},10^{-3},10^{-2},5\times10^{-2},10^{-1}\}
$$
and report the run with the lowest final objective.
%\AI{Note: the selected Adam rate is \(0.1\), the boundary of the stated grid, in
%every supplied plot. Please extend the grid until the selected value is
%interior, state whether bias correction is used, and add standard NAG plus one
%adaptive baseline such as AdaGrad, AMSGrad, or FLAG/FLARE. Otherwise replace
%the phrase \emph{best-tuned Adam} by \emph{best Adam run in the stated grid}.}

All primary and auxiliary variables are initialized at zero. The initial
	adaptive matrices are \(V_0=0\) for Adam and
	\(
	P_0=0.05\|\nabla f(x_0)\|I/\sqrt d
	\)
	for {\it Adam-HNAG} and {\it Adam-HNAG-s}. For Adam, the inverse preconditioner is implemented as
	\((\sqrt{V_k}+\varepsilon I)^{-1}\). For all three adaptive methods, the stabilization constant is
	\(\varepsilon=10^{-8}\).

Table~\ref{tab:parameter-regimes} compares the theorem, the quadratic verification, and the practical logistic-regression experiments.
\begin{table}[!htbp]
	\centering
	\scriptsize
	\setlength{\tabcolsep}{5pt}
	\renewcommand{\arraystretch}{1.25}
		\caption{\YY{Parameter choices in the theorem, quadratic verification, and logistic-regression experiments.}}
	\label{tab:parameter-regimes}
	\begin{tabular}{p{0.13\textwidth}p{0.29\textwidth}p{0.23\textwidth}p{0.23\textwidth}}
		\toprule
			Parameter & Theorem-prescribed choice & Quadratic verification & Practical logistic experiments \\
			\midrule
			Step-size rule &
			$\eta_k=\bar\eta(P_{k-1}^{-1},\nabla f(x_k))$ for Adam-HNAG and $\bar\eta(P_k^{-1},\nabla f(x_k))$ for Adam-HNAG-s; $0<\alpha_k\le\sqrt{\eta_k/2}$ &
			theorem rule; after rejection, $\alpha_k\leftarrow\sqrt{\eta_{k+1}^{\rm tr}/2}$ with $\eta_k$ fixed &
			stabilized formula with no inner loop \\[0.4ex]
			$\gamma_k$ &
			$\alpha_k/R^2$ for Adam-HNAG and $\widetilde{\alpha}_k/R^2$ for Adam-HNAG-s &
			theorem choice &
			same formula \\[0.4ex]
			$R$ &
			a priori trajectory bound &
			$2\|y_0-x^\star\|_\infty$, cf. Fig.~\ref{fig:laplacian-theory-aligned} &
			$R=2$ as a scale parameter \\[0.4ex]
		$P_0$ &
			any diagonal positive definite matrix &
			$0.05\|\nabla f(x_0)\|I/\sqrt d$ &
			$0.05\|\nabla f(x_0)\|I/\sqrt d$ \\[0.4ex]
			$\mathrm{InnerLoop}$ &
			 $\mathrm{True}$&
			$\mathrm{True}$ &
			$\mathrm{False}$ \\[0.4ex]
		$\varepsilon$ &
		$0$ &
			$0$ &
		$10^{-8}$ \\[0.4ex]
			Reported point &
			$x_k^+$ &
			$x_k^+$ &
			$x_k$ \\[0.4ex]
			Consistency condition &
			enforced by inner loop &
			enforced by inner loop &
			monitored a posteriori in Fig.~\ref{fig:lyapunov-condition} \\
		\bottomrule
	\end{tabular}
\end{table}

\begin{figure}[h]
	\centering
	\begin{subfigure}{0.455\textwidth}
		\centering
		\includegraphics[width=\linewidth]{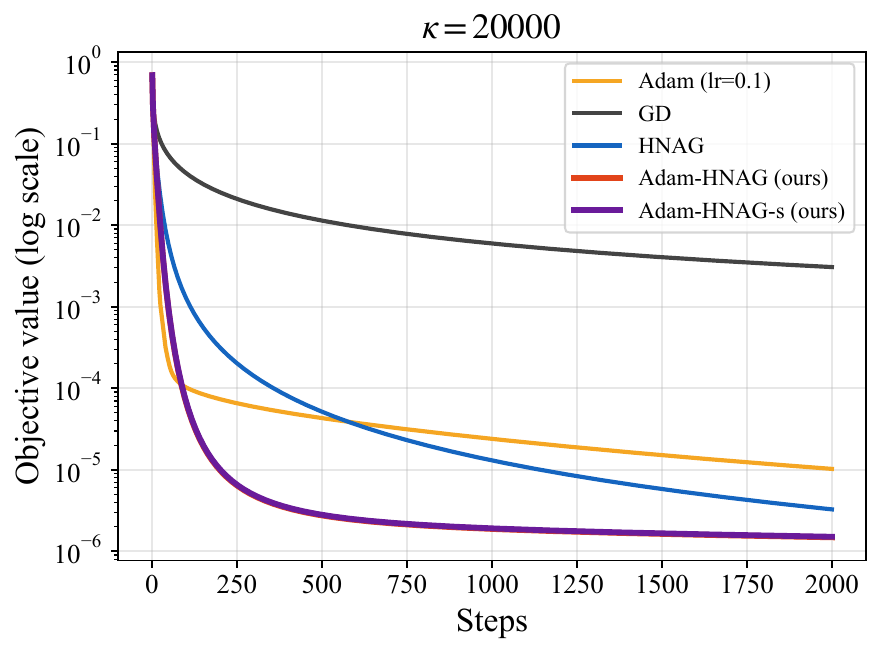}
	\end{subfigure}\hfill
	\begin{subfigure}{0.455\textwidth}
		\centering
		\includegraphics[width=\linewidth]{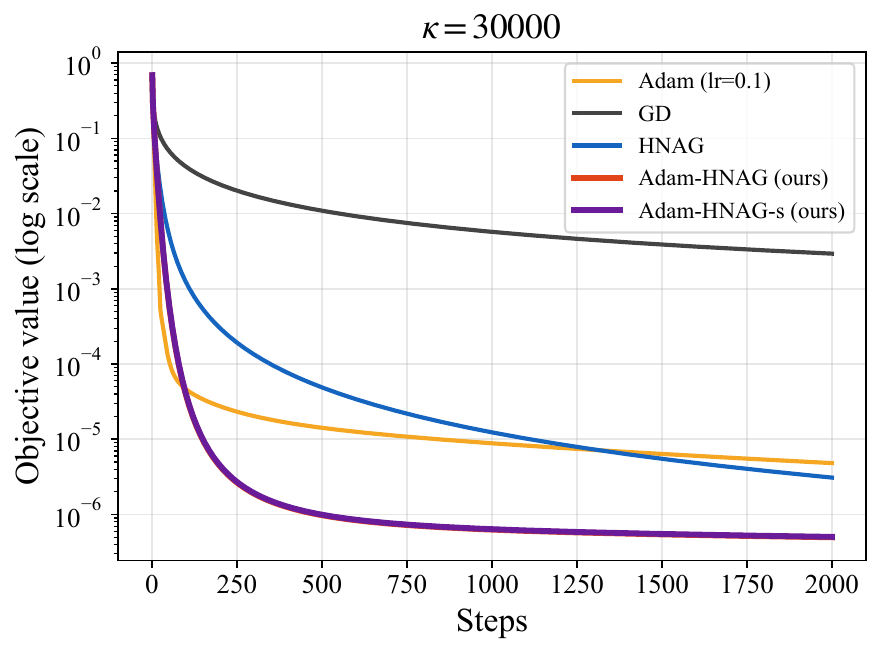}
	\end{subfigure}
	
	\medskip
	
	\begin{subfigure}{0.455\textwidth}
		\centering
		\includegraphics[width=\linewidth]{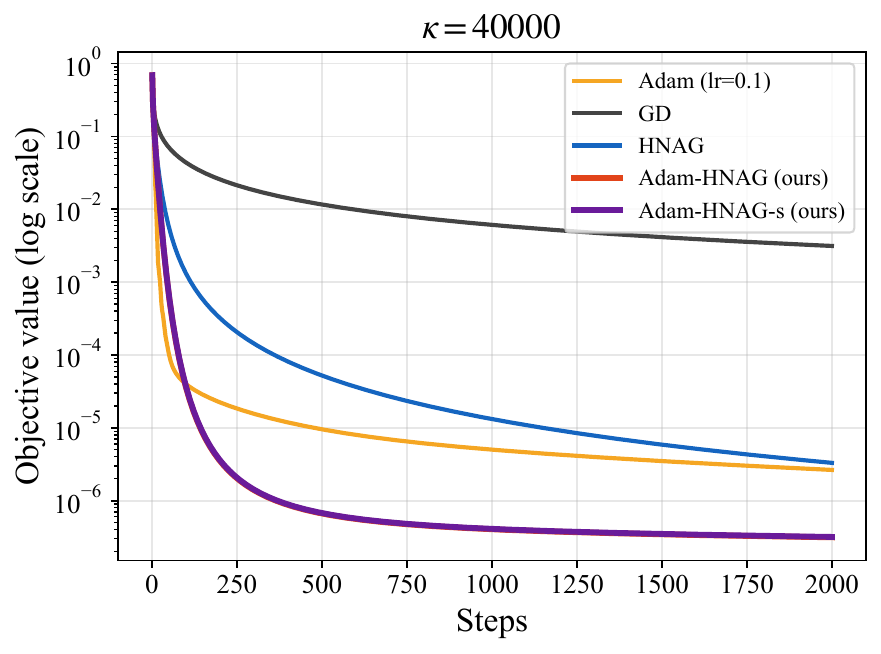}
	\end{subfigure}\hfill
	\begin{subfigure}{0.455\textwidth}
		\centering
		\includegraphics[width=\linewidth]{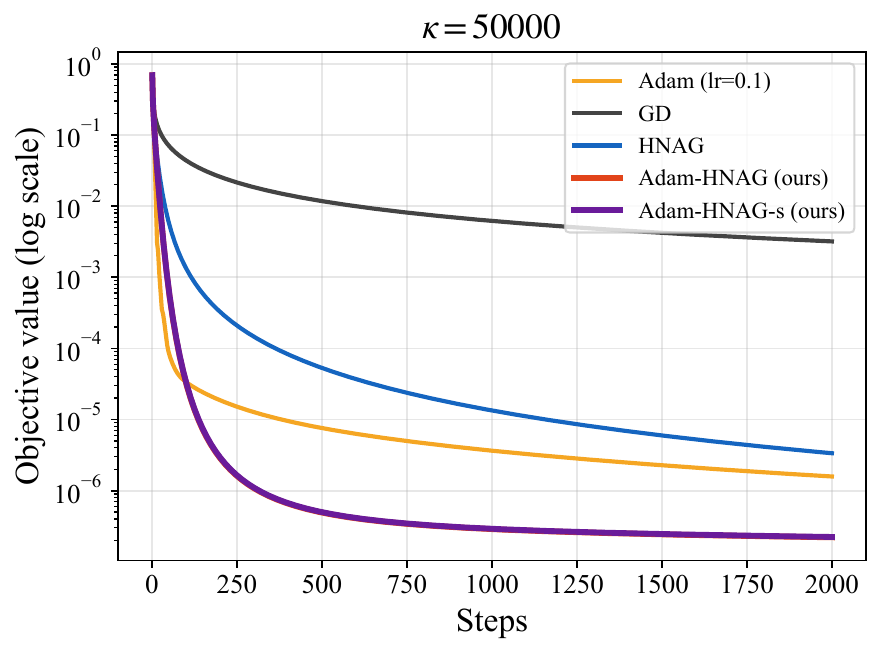}
	\end{subfigure}
	\caption{\YY{Comparison of the optimization methods across different prescribed spectral scalings \(\kappa\).}}
	%\caption{\LC{Comparison across prescribed values of $\kappa$. The spectral scale, margins, labels, and smoothness constant also change, so the panels do not isolate the effect of conditioning.}}
	\label{fig:different-kappa}
\end{figure}%\AI{Note: the supplied panels show only the selected Adam run and use mostly
%solid curves. Please regenerate the figures with print-safe method-specific
%line styles. If nonselected Adam rates are added, show them as thin dotted
%curves and describe only the delivered curves in the response letter.}

\subsubsection{Synthetic data with different spectral scalings.}

We compare datasets generated with different prescribed values of
	$\kappa$. This experiment does not isolate conditioning because the
	construction also changes the spectral scale, smoothness constant, margins,
	and labels.
Specifically, we set
\[
X=U\Sigma V^\top,
\]
where $U\in\mathbb{R}^{n\times d}$ has orthonormal columns and
	$V\in\mathbb{R}^{d\times d}$ is orthogonal. Both matrices are obtained by
	QR factorization of standard Gaussian matrices. We define
	$\Sigma=\operatorname{diag}(\sigma_1,\ldots,\sigma_d)$, where
	\[
	\sigma_i
	=
	1+\frac{i-1}{d-1}\bigl(\sqrt{\kappa}-1\bigr),
	\qquad i=1,\ldots,d.
	\]
	Thus,
	$\operatorname{cond}(X)=\sqrt{\kappa}$ and
	$\operatorname{cond}(X^\top X)=\kappa$.
Since $\sigma_{\max}=\sqrt{\kappa}$, increasing $\kappa$ also
	increases $\|X\|_2$ and the smoothness constant $L$.
We consider
\[
\kappa\in\{20000,30000,40000,50000\}.
\]
The binary labels are generated by
	\[
	\begin{aligned}
		y_i&=\mathbf{1}\{x_i^\top w+\xi_i>0\},\\
		w&=0.1z,\qquad z\sim\mathcal{N}(0,I_d),\\
		\xi_i&=0.3\zeta_i,\qquad \zeta_i\sim\mathcal{N}(0,1),
	\end{aligned}
	\]
	where $x_i^\top$ is the $i$-th row of $X$.
We use the same realizations of $U$, $V$, $z$, and
	$\{\zeta_i\}_{i=1}^n$ for every $\kappa$. Changing $\Sigma$ nevertheless
	changes $X$, the margins, and the resulting labels.
For each fixed
	$\kappa$, all methods are evaluated on exactly the same dataset. Accordingly,
	comparisons across $\kappa$ are interpreted as performance across a family of
	jointly changing spectral instances, rather than as causal evidence for the
	effect of conditioning alone. In all experiments, \(n=500\) and \(d=200\), and each method is run for \(2000\) full-gradient iterations.
%\AI{Note: A controlled conditioning study would instead fix
	%$\sigma_{\max}$, decrease $\sigma_{\min}$, preserve the labels or margins,
	%and report results over several specified random seeds.}

In the early stage, the proposed methods reduce the objective rapidly. The selected Adam run is competitive over the first few hundred iterations, whereas {\it Adam-HNAG} and {\it Adam-HNAG-s} attain lower values later. HNAG provides the Hessian-driven correction, while Adam provides feedback through adaptive preconditioning. Adam-HNAG combines both mechanisms.

\subsubsection{Empirical check of the consistency condition}
In the logistic-regression experiments, the stepsize inner loop is disabled. We therefore check \emph{a posteriori} whether the key inequality in Theorems~\ref{thm:lyapunov-contraction} and~\ref{thm:lyapunov-contraction2} holds along the computed trajectories.

\begin{figure}[!htbp]
	\centering
	\begin{subfigure}{0.455\textwidth}
		\centering
		\includegraphics[width=\linewidth]{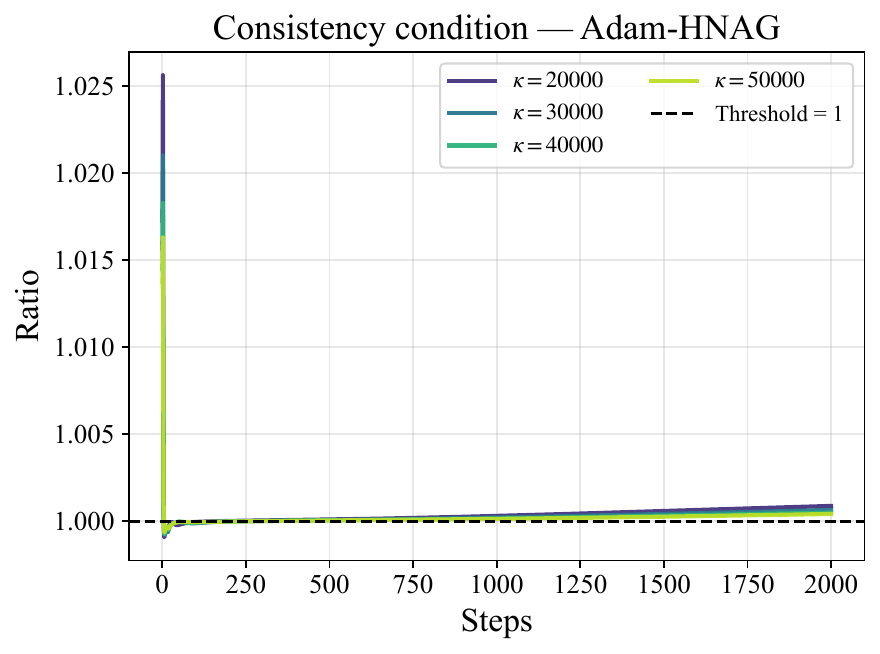}
	\end{subfigure}\hfill
	\begin{subfigure}{0.455\textwidth}
		\centering
		\includegraphics[width=\linewidth]{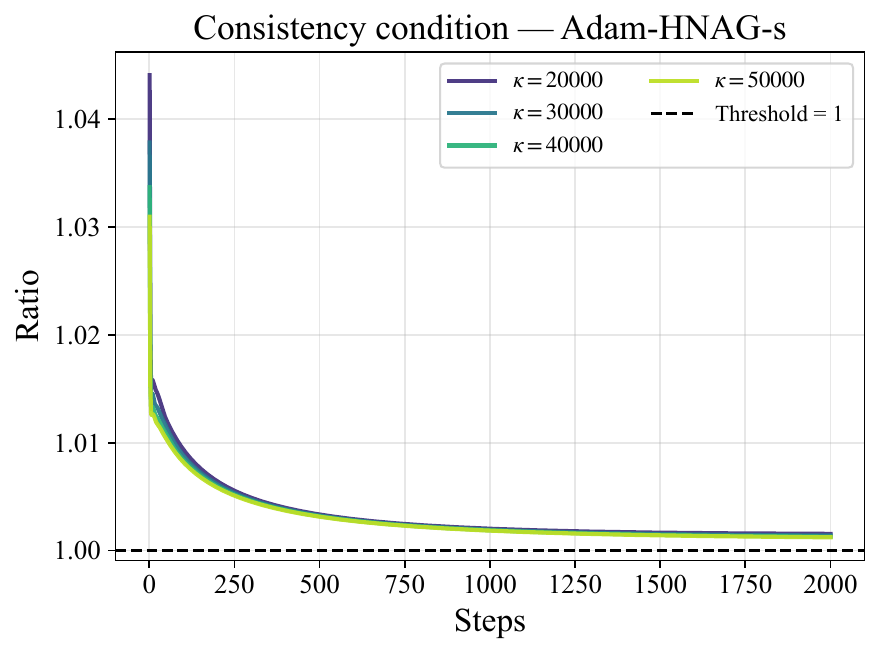}
	\end{subfigure}\hfill
	\caption{
		Empirical evaluation of the consistency condition for {\it Adam-HNAG} (left) and {\it Adam-HNAG-s} (right) across different condition numbers $\kappa$.
		The consistency condition requires the ratio to be not less than $1$.
	}
	\label{fig:lyapunov-condition}
\end{figure}

For both schemes, let
\[
\eta_k=\bar{\eta}(P_{k-\delta}^{-1},\nabla f(x_k)),
\qquad
\alpha_k=\sqrt{\eta_k/2},
\]
where
\[
\delta=
\begin{cases}
	1, & \text{for {\it Adam-HNAG}},\\
	0, & \text{for {\it Adam-HNAG-s}}.
\end{cases}
\]
Then the consistency condition can be written uniformly as
\begin{equation}
	\mathrm{Ratio}
	=
	\frac{\eta_{k+1}(1+\alpha_k)^{2-\delta}}{2\alpha_k^2}
	\ge 1.
\end{equation}
%Equivalently,
%\[
%\mathrm{Ratio}
%=
%\begin{cases}
%\dfrac{\eta_{k+1}(1+\alpha_k)}{2\alpha_k^2}, & \text{for {\it Adam-HNAG}},\\[1.2ex]
%\dfrac{\eta_{k+1}(1+\alpha_k)^2}{2\alpha_k^2}, & \text{for {\it Adam-HNAG-s}}.
%\end{cases}
%\]

The empirical ratios are reported in Figure~\ref{fig:lyapunov-condition}.
{\it Adam-HNAG} exhibits violations of the threshold $1$ during the initial
iterations; therefore the contraction theorem does not apply from \(k=0\) to
these no-inner-loop runs. The ratio later rises above $1$. {\it Adam-HNAG-s}, by
contrast, satisfies the condition throughout, with no violation observed along the
entire trajectory. This is consistent with the weaker consistency requirement of
{\it Adam-HNAG-s}, whose condition involves the larger factor $(1+\alpha_k)^2$.

%\AI{These plots are diagnostics, not a proof for the stabilized implementation.
%Please report the minimum ratio, the number of violations, and the first index
%after which the condition remains satisfied. A contraction may be restarted
%from that index with a new Lyapunov constant, but the earlier steps cannot be
%included in the same telescoping estimate.}

\subsubsection{Real-world data}

We use the \texttt{colon-cancer} dataset from the LIBSVM repository~\cite{Chang2011LIBSVMAL}, with $n=62$ samples and $d=2{,}000$ features. Since $d\gg n$, the problem is underdetermined and poorly conditioned under the small regularization used here. We use the data without additional preprocessing and run each method for $500$ full-gradient iterations with the same hyperparameters as in the synthetic experiments.

\begin{figure}[!htbp]
	\centering
	\begin{subfigure}{0.455\textwidth}
		\centering
		\includegraphics[width=\linewidth]{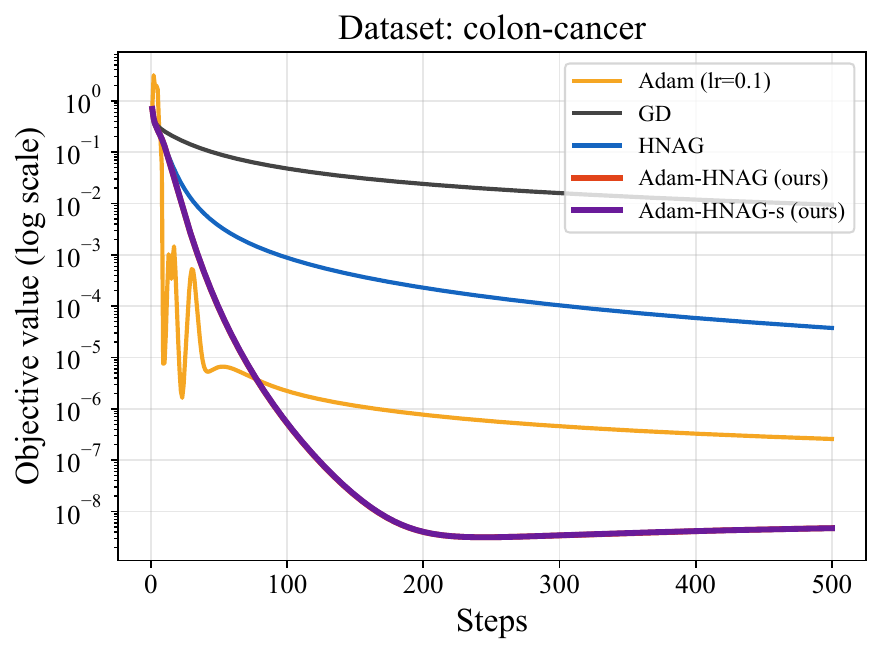}
	\end{subfigure}\hfill
	\begin{subfigure}{0.455\textwidth}
		\centering
		\includegraphics[width=\linewidth]{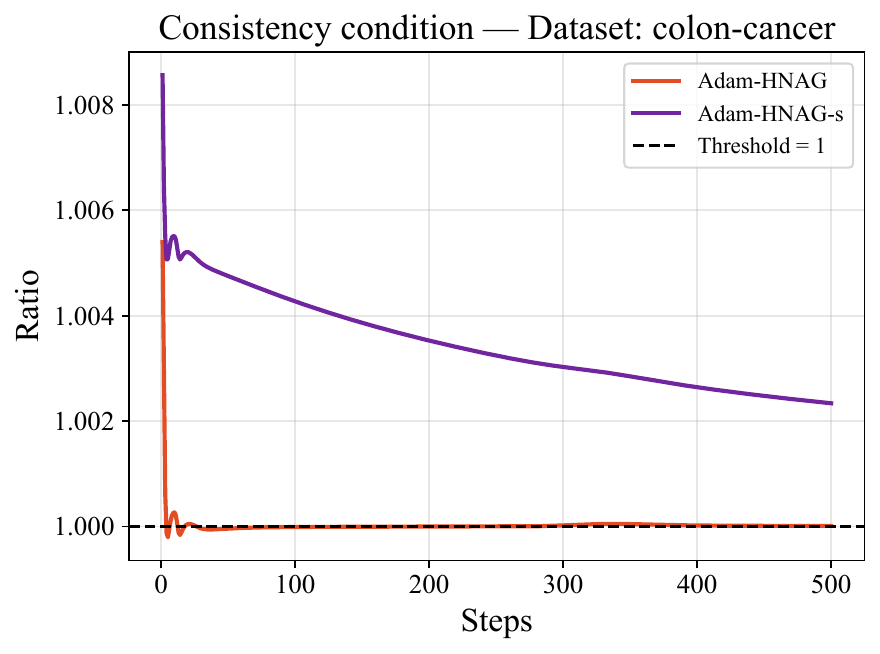}
	\end{subfigure}\hfill
	\caption{
		Left: regularized objective value on \texttt{colon-cancer}. 
		Right: empirical evaluation of the consistency condition on 
		\texttt{colon-cancer}. The consistency condition requires 
		the ratio to be not less than $1$.
	}
	\label{fig:real_data_colon-cancer}
\end{figure}

Figure~\ref{fig:real_data_colon-cancer} (left) shows the regularized objective. Within the stated iteration budget, {\it Adam-HNAG} and {\it Adam-HNAG-s} attain lower values than GD, HNAG, and the selected Adam run. Both methods reach values of order $10^{-8}$ after about $200$ iterations, whereas the baselines remain above this level after $500$ iterations. On this dataset, the results indicate that the proposed adaptive preconditioning remains effective in the high-dimensional regime $d\gg n$.

Figure~\ref{fig:real_data_colon-cancer} (right) reports the consistency diagnostic. {\it Adam-HNAG} shows small initial violations before the ratio stabilizes above the threshold $1$, whereas {\it Adam-HNAG-s} satisfies the condition throughout.

Across the fixed-objective experiments, {\it Adam-HNAG} and {\it Adam-HNAG-s} produce nearly identical curves. These tests therefore do not distinguish the two update orders. The next online example lies outside the theory and serves only as a stress test in which their behavior differs.

\subsection{A classical counterexample for Adam}
Finally, we include the classical synthetic counterexample of Reddi et al.~\cite{reddi2019convergenceadam}. Although this example arises in the online convex optimization setting described in \eqref{eq:online-regret}, it is used here only as a controlled stress test for Adam-type methods in a regime where Adam is known to behave pathologically. Following~\cite{reddi2019convergenceadam}, the one-dimensional convex domain $\mathcal F=[-1,1]$ and the loss sequence
$$
f_t(x)=
\begin{cases}
	1010x, & \text{for} \quad t \bmod 101 = 1,\\
	-10x, & \text{otherwise}
\end{cases}
$$
are considered. The optimal solution is $x^\star=-1$.

Adam, AMSGrad~\cite{reddi2019convergenceadam}, {\it Adam-HNAG}, and {\it Adam-HNAG-s} are compared. Since this online benchmark does not come with a meaningful smoothness constant $L$, the quantity $\bar{\eta}(P^{-1},\nabla f(x))$ used in the convex convergence theory is not directly applicable.
Moreover, the theoretical couplings among $\alpha_k$, $\beta_k$, and $\gamma_k$ were derived for the fixed smooth convex setting and are not tailored to this adversarial online example. We therefore adopt a decoupled practical parametrization: the same metric-ratio form is used for $\alpha_k$, while $\beta$ and $\gamma$ are tuned independently. Specifically, we set
\[
\alpha_k
=
\eta \sqrt{
	\frac{\nabla f(x_k)^\top P_{k-\delta}^{-1}\nabla f(x_k)}
	{\nabla f(x_k)^\top P_{k-\delta}^{-2}\nabla f(x_k)}
},
\]
and use the same hyperparameters $(\eta,\beta,\gamma)=(10^{-3},0.01,0.1)$ for both {\it Adam-HNAG} and {\it Adam-HNAG-s}, with the switch parameter $\delta=1$ for {\it Adam-HNAG} and $\delta=0$ for {\it Adam-HNAG-s}.
For Adam and AMSGrad, the choice in~\cite{reddi2019convergenceadam},
$$
(\eta,\beta_1,\beta_2)=(0.01,0.9,0.99),
$$
is used.

All methods are initialized at \(x=0\), with the auxiliary and moment
variables initialized at zero. For the proposed methods, consistently with
the other experiments, we set \(P_{-1}=P_0=0.05|\nabla f_1(x_0)|\) in this one-dimensional example. We use \(\varepsilon=10^{-8}\) for numerical stability.

%\AI{Note: please state the projection onto \([-1,1]\), initialization, stabilization, horizon, and exact regret calculation, and include a sensitivity check for \((\eta,\beta,\gamma)\). A fixed smooth-convex batch limit-cycle example~\cite{bock2022cycles} would be a stronger in-scope stress test.}

\begin{figure}[!htbp]
	\centering
	\begin{subfigure}{0.455\textwidth}
		\centering
		\includegraphics[width=\linewidth]{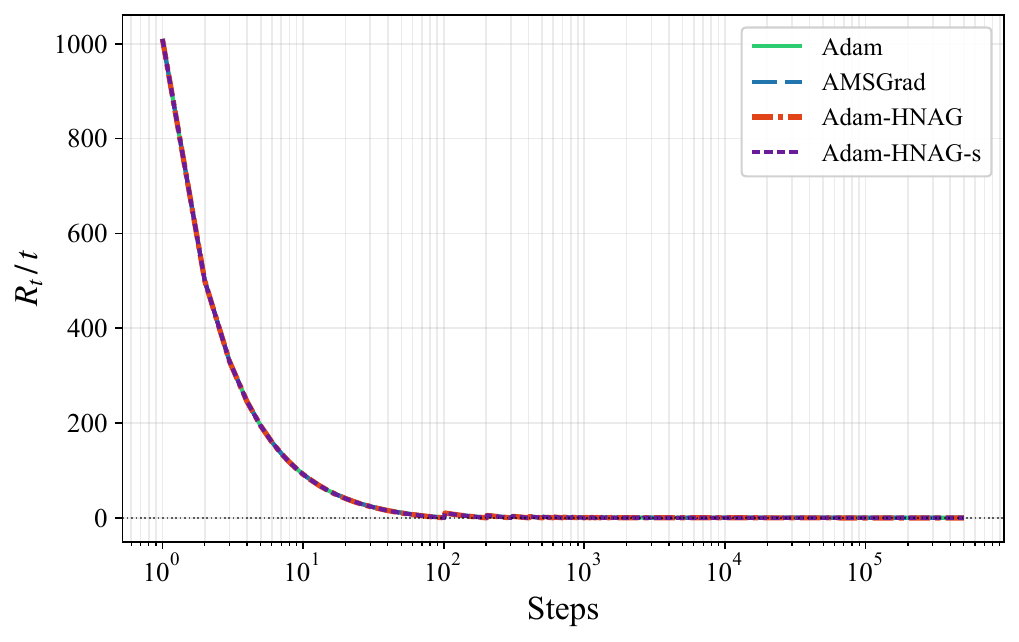}
	\end{subfigure}\hfill
	\begin{subfigure}{0.455\textwidth}
		\centering
		\includegraphics[width=\linewidth]{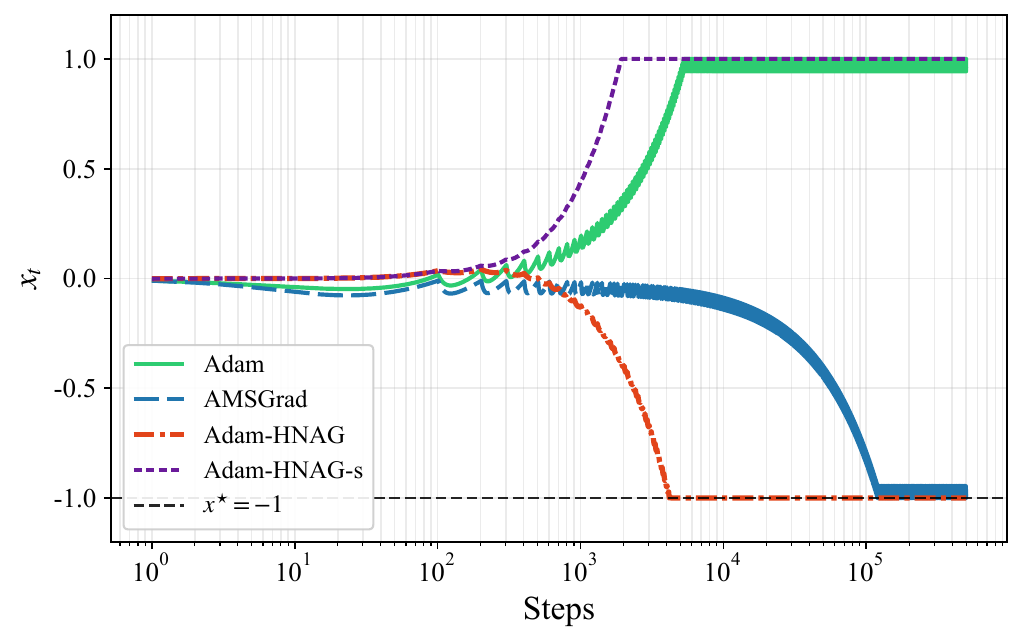}
	\end{subfigure}\hfill
	\caption{Classical synthetic counterexample of Reddi et al.~\cite{reddi2019convergenceadam}. Left: average regret \(R_t/t\). Right: iterate trajectory \(x_t\).}
	\label{fig:redditest}
\end{figure}

This benchmark is a stress test, not a validation of the convergence theory. The theory concerns a fixed smooth convex objective, whereas the construction of Reddi et al.~\cite{reddi2019convergenceadam} uses a time-varying sequence of online losses. We report both the trajectory $x_t$ and the average regret. Under the stated parameters, {\it Adam-HNAG} remains stable, whereas Adam and the synchronous variant {\it Adam-HNAG-s} drift toward the nonoptimal endpoint. In this experiment, {\it Adam-HNAG} also reaches the correct endpoint faster than AMSGrad, the remedy proposed in~\cite{reddi2019convergenceadam}. This single stress test shows that the methods can behave differently, but it does not support a general stability comparison.

\section{Conclusion}\label{sec:conclusion}

We introduced an Adam-inspired flow by approximately splitting the momentum and square-gradient metric dynamics and adding the gradient correction used in HNAG. This construction yields two deterministic methods, {\it Adam-HNAG} and {\it Adam-HNAG-s}. For their exact unstabilized forms, we prove a product Lyapunov contraction under a trajectory bound and a one-step consistency condition; the stepsize inner loop enforces the latter. When the exact adaptive steps are accepted, the product estimate gives an $O(k^{-2})$ objective-value bound for both methods. The gradient-magnitude feedback in the preconditioner update distinguishes these methods from HNAG and standard accelerated gradient methods.

The main contribution is a Lyapunov coupling between a changing diagonal metric and a Hessian-driven accelerated correction. Natural next steps are to remove the trajectory bound and extend the analysis to stochastic and nonconvex problems under assumptions that can be verified in practice.

\FloatBarrier
\bibliographystyle{elsarticle-num}
\bibliography{references}

% \AI{The following entries are cited by the newly added related-work paragraph
% in the introduction. Please move them into references.bib.}
% @inproceedings{levy2018online,
%   title={Online adaptive methods, universality and acceleration},
%   author={Levy, Kfir Y. and Yurtsever, Alp and Cevher, Volkan},
%   booktitle={Advances in Neural Information Processing Systems (NeurIPS)},
%   year={2018}}
% @inproceedings{kavis2019unixgrad,
%   title={{UniXGrad}: A universal, adaptive algorithm with optimal guarantees
%          for constrained optimization},
%   author={Kavis, Ali and Levy, Kfir Y. and Bach, Francis and Cevher, Volkan},
%   booktitle={Advances in Neural Information Processing Systems (NeurIPS)},
%   year={2019}}
% @inproceedings{joulani2020simpler,
%   title={A simpler approach to accelerated optimization: iterative averaging
%          meets optimism},
%   author={Joulani, Pooria and Raj, Anant and Gy{\"o}rgy, Andr{\'a}s and
%           Szepesv{\'a}ri, Csaba},
%   booktitle={International Conference on Machine Learning (ICML)},
%   year={2020}}
% @inproceedings{ene2021adaptive,
%   title={Adaptive gradient methods for constrained convex optimization and
%          variational inequalities},
%   author={Ene, Alina and Nguyen, Huy L. and Vladu, Adrian},
%   booktitle={AAAI Conference on Artificial Intelligence},
%   year={2021}}

\end{document}